\documentclass[11pt,a4paper]{article}

\usepackage[a4paper,top=3.8cm,bottom=3.8cm,left=3cm,right=3cm]{geometry}

\usepackage[T1]{fontenc}
\usepackage{amsmath,amssymb,mathtools,amsthm,amsopn}
\usepackage{hyperref}

\usepackage{mathpazo}
\usepackage[dvipsnames]{xcolor}

\usepackage{titlesec}
\titleformat{\section}{%
\normalfont\large\bfseries}{\thesection.}{1em}{}
\titleformat{\subsection}{%
\normalfont\normalsize\bfseries}{\thesubsection.}{1em}{}

\usepackage{enumitem}

\usepackage{tocloft}
\usepackage{fancyhdr}
\usepackage{fourier-orns}

\usepackage{cite}

\renewcommand{\phi}{\varphi}
\newcommand{\scalar}[2]{\langle#1,#2\rangle}

\DeclareMathOperator{\Int}{int}

\newcommand{\step}[1]{\par\medskip\noindent\it#1\rm}

\newcommand{\abs}[1]{\lvert#1\rvert}

\newcommand{\z}{\zeta}

\newcommand{\wh}{\widehat}
\renewcommand{\H}{\mathbb{H}}
\newcommand{\e}{\varepsilon}
\renewcommand{\r}{\varrho}

\newcommand{\s}{\sigma}

\newcommand{\la}{\lambda}

\newcommand{\ol}{\overline}

\renewcommand{\d}{\delta}
\newcommand{\Eucl}{\textup{Euc}}
 
\newcommand{\p}{\partial}

\newcommand{\R}{\mathbb{R}}
\newcommand{\C}{\mathbb{C}}
\newcommand{\G}{\mathbb{G}}

\newcommand{\N}{\mathbb{N}}

\newtheoremstyle{pippo}  
  {}       
  {}       
   {\sffamily}   
 {}        
  {\bfseries}  
  {.}   
  {1ex}       
  {}           
\newtheoremstyle{pluto}  {}{}
{\slshape}  {}{\bfseries}  {.} {1ex}    {}

 \newtheorem{theorem}{Theorem}[section]

\newtheorem{proposition}[theorem]{Proposition}

\newtheorem{lemma}[theorem]{Lemma}

\theoremstyle{pluto} 
 \newtheorem{definition}[theorem]{Definition}

\newtheorem{remark}[theorem]{Remark}
\renewcommand{\d}{\delta}
\renewcommand{\t}{\tau}

\renewcommand{\a}{\alpha}
\renewcommand{\b}{\beta}

\DeclareMathOperator{\Span}{span}

\usepackage[dayofweek]{datetime}

\numberwithin{equation}{section}

\begin{document}

\title{On the inner cone property for convex sets in two-step Carnot groups, with applications to monotone sets
\thanks{2010 Mathematics Subject Classification. Primary 53C17; Secondary
 52A01.
 Key words and Phrases.     SubRiemannian distance. Carnot groups. Monotone sets.}
}
\author{Daniele Morbidelli
  \thanks{Dipartimento di Matematica, Universit\`a di Bologna (Italy)}}

\date{\today}

\maketitle

 \begin{abstract}\small
 In the setting of step two Carnot groups, we show a ``cone property'' for horizontally convex sets. Namely 
 we prove that, given a horizontally  convex set~$C$,  a pair of  points~$P\in \p C$ and $Q\in \Int (C)$, both 
 belonging to a horizontal line~$\ell$,
then an open truncated subRiemannian cone around~$\ell$   and with vertex at~$P$ is contained in~$C$.

We apply our result to the problem of classification of horizontally monotone sets in Carnot groups. We are able to show that monotone sets in the direct product~$\H\times\R$ of the Heisenberg group with the real line have hyperplanes as boundaries.
 \end{abstract}


\section{Introduction and main results}  
The starting point of this paper is a result of Arena, Caruso and Monti\cite{ArenaCarusoMonti12}, where it is proved that in  subRiemannian two-step Carnot groups of M\'etivier type,  given a horizontally  convex set~$C$,  a pair of  points~$P\in \p C$ and $Q\in \Int (C)$, both belonging to a horizontal line~$\ell$,
then an open truncated subRiemannian cone around~$\ell$ and with vertex at~$P$ is contained in~$C$. Here we   prove that the same property holds in general two-step Carnot groups, without  assuming the restrictive M\'etivier condition.

We apply the construction above to the problem  of classification of \emph{precisely monotone} sets, in the sense of Cheeger and Kleiner \cite{CheegerKleiner10}. Recall that a  precisely monotone  set $E$ in a Carnot group~$\G$ is a set which is (horizontally) convex and  such that $\G\setminus E $ is convex too. 
In~\cite{CheegerKleiner10} it is shown that a precisely monotone set~$E$ in the three-dimensional Heisenberg group~$\H$ satisfies $\Pi\subset E\subset\ol \Pi$, for  a suitable  half-space~$\Pi$. The classification is nontrivial, see also~\cite{CheegerKleinerNaor11}.  

In this paper we show the same result in  $\H\times \R$, 
the direct product of~$\H$ with the Euclidean line.
Namely, we will prove that  the boundary of  a (precisely) monotone set in~$\H\times\R$ is a hyperplane. See Theorem~\ref{massimino}
 for the precise statement. 
 Although our result concerns the seemingly easy model~$\H\times\R$, the proof   requires   a certain amount of work.    Furthermore, our techniques, based 
 on the use of the cone property above, connected with the notion of \emph{intrinsic graph}, in the sense of 
 \cite{FSSC01}, or \cite{AmbrosioSerraCassanoVittone06}, could be useful to attack the analogous classification problem 
 for  general two-step Carnot groups. Let us mention that at the author's knowledge, such problem is at the 
 moment open.

To describe our setting, let $(z,t)\in Z\times T=\R^m\times\R^{\ell} =\G$ be equipped with  the group law
\begin{equation}\label{machegruppo} 
 (z,t)\cdot (\z,\t)=(z+\z, t+\t+\scalar{ z}{A \z})
 =(z+\z, t+\t+Q(z,\z))
\end{equation} 
where $A=(A^{ 1}, \dots, A^\ell)$  and   $A^\b\in\R^{m\times m}$ is a skew-symmetric matrix for all $\b=1,\dots, \ell$. In other words, $Q:Z\times Z\to T$ is bilinear and skew-symmetric.
  The left invariant horizontal vector fields
\begin{equation}
X_j:=\p_{z_j}+\sum_{ \b=1}^{\ell}\sum_{k=1}^{ m}a_{kj}^\b z_k\p_{t_\b},\qquad j=1,\dots, m,
\end{equation} 
define in a standard way a subRiemannian distance~$d$. 
 We denote by $B((x,t), r)$ the corresponding  ball of center $(z,t)$ and radius $r$. See~\cite{NagelSteinWainger} for the definition. 
The \emph{horizontal plane} at a point $(z,t)$ will be denoted by
$H_{(z,t)}=\Span\{X_j(z,t) :1\leq j\leq m\}$.   
Furthermore, for any  $u\in\R^m$, letting   $u\cdot X=\sum_{k\le m}u_jX_j$,   
 we see that the integral curve of     $u\cdot X$ at time~$1$ starting from~$(z,t)$ has the form~$
 e^{u\cdot X}(z,t)= (z+u, t+Q(z,u))
$.  

A \emph{horizontal line}, briefly, a \emph{line}, is a set of the form 
\begin{equation}\label{elline} 
\ell= \{ e^{s u\cdot Z}(z,t)=(z,t)\cdot (su,0)=(z+su, t+sQ(z,u)) : s \in\R\}                                                          \end{equation} 
where $u\in\R^m$ and $(z,t)\in\G.$
 \emph{Lines} in our setting are particular Euclidean lines. Note that the Euclidean line parametrized by the path $\gamma(s)=(as,bs,cs)$   is a horizontal line if and only if $c=0$. We always parametrize horizontal lines with constant speed.

Finally, 
we say that a set  $C\subset \G$ is \emph{horizontally convex}, or briefly, \emph{convex} if for all $P,Q$ (horizontally) aligned points contained in $C$, the whole line containing $P$ and $Q$ is contained in $C$. We say instead that~$E\subset\G$ is \emph{horizontally precisely monotone}, or briefly, \emph{monotone},  if both $E$ and $\G\setminus E$ are convex. It is well known that the class of horizontally convex sets includes strictly all  sets which are convex in the Euclidean sense.

   Convexity in Carnot groups has been introduced in  \cite{DanielliGarofaloNhieu03} and   
   \cite{LuManfrediStroffolini}. Further references on properties of convex sets are 
   \cite{MontiRickly05,CalogeroCarcanoPini07,CalogeroPini}. Many authors have studied regularity properties 
   of convex functions.
   Monotone sets  in the Heisenberg group appear and have a prominent role in the recent papers
    \cite{CheegerKleiner10,CheegerKleinerNaor11,NaorYoung,FasslerOrponenRigot}. Observe also that the class 
    of monotone sets is somewhat similar to the more restricted class of \emph{sets with constant horizontal 
    normal}, appearing in geometric measure theory in Carnot groups. See 
    \cite{FSSC01,FSSC03,BellettiniLedonne12}.

  Here is our first result.
\begin{theorem}
 \label{conogelato}Let $C\subset \G$ be a convex set in a two-step Carnot group with law~\eqref{machegruppo}. 
 Let    $P=(z, t )\in\overline{ C}$   and assume that there is $\xi\in Z=\R^m$ such that $Q:=(z , t )\cdot(\xi,0)\in \Int 
 (C)$. Then  there is   $\e>0$ such that 
 \begin{equation}\label{coppifausto} 
  \bigcup_{0<s\leq 1} B\Big((z' , t' )\cdot (s\xi,0),\e s\Big)\subset C\qquad\text{ for all $ (z',t')\in B\big((z , t ),\e\big)\cap \ol C$}.
 \end{equation} 
\end{theorem}
   An analogous statement holds  without assuming  that  $(z_, t )\cdot (\xi,0)\in \Int (C)$, but requiring that 
   a  surface $\Sigma$  containing  $Q$ and transversal to the line containing $P$ and $Q$    is contained 
   in $C$.     Note also that the statement holds for $P=(z,t) $ in the closure of $C$ and becomes more significant when $(z,t)\in \p C$. Even the situation $P\in  \p C  \setminus C$ is included. In this case, the mere definition of convexity does not  ensure that the open horizontal segment connecting  $(z,t)$ with $(z,t)\cdot (\xi,0)$ is contained in $C$. However, the  proof of the  theorem shows that such segment is contaned in  $\Int(C)$.

The theorem above was already proved by Arena, Caruso  and Monti \cite[Theorem~1.4]{ArenaCarusoMonti12}, in the case of two-step groups of M\'etivier type.  Recall that a group of the form \eqref{machegruppo} satisfies the M\'etivier condition if for all $t\in T$ and for any non-zero $z\in Z$,   there is $\z\in Z$ such that $Q(z,\z)=t$. A similar, more qualitative statemeent was proved by Cheeger and Kleiner~\cite[Proposition~4.6]{CheegerKleiner10}, in the setting of the three-dimensional Heisenberg group.

In the present paper, we generalize the statement to general  step 2 groups. Our techniques are inspired by the argument in \cite{CheegerKleiner10}, but we exploit some higher-order  envelopes, see Section~\ref{ordine}. Namely, the key starting point of the classification argument of \cite{CheegerKleiner10}  relies on the analysis of the map
\begin{equation*}
Z\times Z\ni(u_1, u_2)\mapsto \Gamma(u_1, u_2):= e^{u_2\cdot X}e^{u_1\cdot X}(0,0),
\end{equation*}
which turns out to be a submersion at any point $(\xi,\xi)\neq (0,0)$ (the argument  in  \cite{ArenaCarusoMonti12} is essentially based on a similar property). This property fails to hold if the group does not satisfy the M\'etivier condition (see the discussion in Section~\ref{ordine}). We are then forced to analyze ``higher order'' maps of the form 
\begin{equation}\label{king} 
 Z^p\ni (u_1, u_2, \dots, u_p)\mapsto \Gamma(u_1, u_2, \dots, u_p)\mapsto e^{u_p\cdot X}\cdots e^{u_2\cdot X} e^{u_1\cdot X}(0,0),
\end{equation}
where $p$ can be greater than  $2$. It may happen that at some $(\xi,\xi,\dots, \xi)\neq 0\in Z^p$ there is no $p\in\N$ such that the map $\Gamma$ is a submersion.  However, we are able to show that, given any two-step Carnot group,   for sufficiently large $p$ the map $\Gamma $ is  \emph{open} at any point of the form 
$(\xi,\xi,\dots, \xi) $. See the statement in Theorem~\ref{opp} and see also Remark~\ref{presto}. The existence of the cone in~\eqref{coppifausto} follows  easily.

We apply our construction to the study in  a model case of the classification problem for monotone subsets 
  of a Carnot group. It is easy to see that in any Carnot group of step two, given any hyperplane~$\Sigma$, 
  any of the two open half spaces    whose union gives $\G\setminus\Sigma$ is a monotone set. The natural 
  question is whether  any monotone set~$E$ has a hyperplane as boundary. If $\G$ is of M\'etivier type, then 
  this turns out to 
  be true   under the mild assumption $\p E\subsetneqq \G$, which is always satisfied if $\G=\H^n$ is the $n$-th Heisenberg group.    The short argument in the proof of Proposition~\ref{decimo} 
  has been kindly  indicated to the author by Roberto Monti.   See Proposition~\ref{vameglio} for the case of~$\H^n$.

   The first significant    example of non M\'etivier two step Carnot group  appears in the following 
   situation. 
Consider  the direct product $ \G:=\H_{(x,y,t)} \times \R_u=\R^4$
with law 
\begin{equation*}
 (x,y,u,t)\cdot (x' , y', u', t')  =   (x+x' ,  y+y',u+u', t+ t' +2(y x'-y'x ))
\end{equation*}
and with horizontal left-invariant vector fields $X=\p_x+2y\p_t$, $Y=\p_y-2x\p_t$ and $U=\p_u$. We shall prove the following theorem.
 \begin{theorem} \label{massimino} 
   Let $E\subset \H\times \R$ be a monotone set. Then there is a three-dimensional hyperplane~$\Sigma$ such 
   that $\Int (E)$
   and $\Int (E^c)$ are   the  connected components of $
   \R^4\setminus\Sigma$. 
  \end{theorem}

  To prove such theorem, keeping in mind the argument of \cite{CheegerKleiner10}, we have to introduce some new ideas, in order to have a starting point which is suitable for a higher-dimensional situation. Namely we 
  will use our result on cones, and the consequent fact that around a point~$P$ at  the boundary of $E$,  
  either the boundary is contained in the horizontal plane $H_P$, or it is contained in an intrinsic graph. 
  See the discussion in Section~\ref{custom74}.   We are then able to classify, first locally and then 
  globally,  intrinsic graphs whose corresponding 
  epigraphs are monotone sets. 
  
  Let us mention that an alternative approach to Theorem~\ref{massimino} could 
  be carried out via a parametric analysis of the intersections of~$E$ with sets with $u$= constant, which are basically 
  copies of the Heisenberg group. This alternative 
  method however would not provide any help in the analysis of more general situations.
 We are choosing   instead our approach with graphs in the perspective of applying  similar ideas to 
  more general two-step Carnot groups. Concerning the situation in higher step, the examples 
  in~\cite{BellettiniLedonne12} suggest that monotone sets
  in higher step Carnot groups could be very ugly.

  The plan of the paper is the following. In Section~\ref{ordine} we discuss the openness property of the 
  map~$\Gamma$ mentioned in~\eqref{king}. In Section~\ref{homo} we apply such property to the construction 
  of the 
  cones  in Theorem~\ref{conogelato}. In Section~\ref{custom74} we show that the 
  boundary of a 
  monotone set  can be written as an intrinsic graph.     Section~\ref{4} is devoted to the classification 
  of monotone sets in the group~$\H\times\R$.

  \section{Higher order envelops,  open maps and cones }
  \label{sapporo} 
\subsection{Multiexponential open maps }

\label{ordine} 
In this section, given  $p\in\N$,  we consider the map $\Gamma:Z^p\to Z\times T$,
\begin{equation}\label{eppi} 
 \Gamma(u_1,\dots, u_p):=e^{u_p\cdot X}\cdots 
 e^{u_2\cdot X}e^{u_1\cdot X}(0)=\Bigl(\sum_{j\le p}u_j,
 \sum_{1\le j<k\le p}Q(u_j,u_k)
 \Bigr).
\end{equation} 
Here we use the notation $u\cdot X=\sum_{k=1}^m \scalar{u}{e_k} X_k$ for a vector $u= (\scalar{u}{e_1} , \dots,\scalar{u}{e_m})\in Z=\R^m$.
Our purpose is to prove that the map $\Gamma$ is open at any point $(\xi,\dots,\xi)$.

\begin{theorem}\label{opp} 
 In any two-step Carnot group with law~\eqref{machegruppo} there is  
 $p\in\N$ such that  for all $\xi\in Z$ the map~$\Gamma$ in~\eqref{eppi}
is open at any point $(\xi,\dots,\xi)\in Z^p$. More precisely, for all $\xi\in Z$ there is $c_0>0$ such that 
$\Gamma (B_\Eucl((\xi, \xi,\dots, \xi), r))\supseteq B_\Eucl(\Gamma(\xi,\dots, \xi),c_0 r^2) $ for all $r\leq c_0$.
\end{theorem}
Here and hereafter, if needed,  $B_\Eucl $ denotes Euclidean balls. From the proof, it will be clear that the number~$p$ can be estimated in terms of the dimension~$m$ of~$ Z$. However we do not need a sharp estimate of~$p$. 
 
\begin{remark} 
\label{presto} 
If the group satisfies the M\'etivier condition, then it suffices 
 to choose $p=2$ and it turns out that for all $\xi\neq 0$ the map $\Gamma$ is a submersion (see 
 \cite{CheegerKleiner10}). For more general groups, we are forced to consider larger 
 values of $p$. Furthermore, the map $\Gamma$ can not  be a submersion at any point $(\xi,\dots, \xi)\in Z^p$,  with 
 any $p\in\N$
if the curve $\gamma(s)= (s\xi, 0)$ is an  abnormal extremal  for the length-minimizing 
 subRiemannian problem (note that the M\'etivier condition characterizes   two step Carnot groups where  
 subRiemannian abnormal  geodesics do not appear at all, 
 see \cite[Proposition~3.6]{MontanariMorbidelli15}). 
 
 In spite of these pathologies,   the ``quadratic'' 
 openness proved in Theorem~\ref{opp}  holds and is sufficient for our purposes.  We observe
 finally that our theorem holds also for $\xi=0$. 
\end{remark}

In order to prove Theorem \ref{opp} we need to show some properties of the following  map.
Define for all $q\in\N$ the  function  $H_{2q}:Z^{2q}\to T$,
\begin{equation*}
H_{2q}  (z_1,\z_1,z_2,\z_2,\dots,z_q,\z_q) :=\sum_{k=1}^qQ(z_k,\z_k)\in T.         
\end{equation*}

\begin{lemma}\label{duetre} 
 Let $Z=\R^m$ and assume the H\"ormander condition. Then there is $q\in \N$ and $c>0$ such that 
\begin{equation}\label{accaqu} 
 H_{2q}(B_\Eucl(0,r))\supseteq B_\Eucl(0, cr^2)\quad\text{for all $r\in\mathopen]0,+\infty\mathclose[.$}
\end{equation}  

 \end{lemma}
 From the proof it is clear that, if $Z$ has dimension $m$, then any $q\geq m-1 $ ensures~\eqref{accaqu}   for some~$c>0$. 
\begin{proof}[Proof of Lemma \ref{duetre}] 
Write $Q(z,\z)=\scalar{z}{A\z}$. The step two  H\"ormander condition ensures that $\Span\{A_{jk}:j<k\}=T$. Since $A_{jk}= Q(e_j,e_k)$, we conclude  that for all $t\in T$ there is a vector  $(c_{jk})_{1\leq j<k\leq m} $  such that
\begin{equation}\label{indipo} 
\sum_{j<k} c_{jk}  Q(e_j, e_k)=t.                                \end{equation} 
Writing the left-hand side in the form $\sum_{k=2}^m Q(\sum_{j=1 }^{k-1}c_{jk}e_j, e_k)$, we see  
that the map $H_{2q}$ with $q=m-1$ is onto.

  Next, to conclude the argument we use the following homogeneity argument. 
Choose 
 $(c_{jk})_{j<k}$ such that~\eqref{indipo} holds with  estimate $\sum\abs{c_{jk}}\leq C \abs{t}$  (this can be done  with $C$ uniform in $t$ by elementary linear algebra).
Then for all $k=2,\dots, m$ we let $w_k:=\sum_{j=1}^{k-1} c_{jk} e_j$ (so that $|w_k|\leq C|t|$)
and we choose
\begin{equation*}
z_{k }:=|w_k|^{-1/2} w_k 
\quad\text{ and }\quad 
\z_k =|w_k|^{1/2}e_k
\end{equation*} 
(if $w_k=0$ then we take $z_k=\z_k=0$). Thus, we get  
$\sum_{k=2}^{m }Q(z_k,\z_k)= t$ with the required estimate 
$|z_k|,|\z_k|\leq C|t|^{1/2}$ for all $k=2,\dots, m $.
\end{proof}
 
 \begin{proof}[Proof of Theorem \ref{opp}]
Let $\xi\in Z$ and define  
\begin{equation*}
\begin{aligned}
 f_\xi(u_1, \dots, u_p):& =\Gamma(\xi+u_1, \dots, \xi + u_p)-    \Gamma(\xi, \dots, \xi)
 \\&= \Big(\sum_{j\leq p}u_j, Q\Big( \sum_{i=1}^p(p-2i+1)u_i,\xi\Big)  + \sum_{1\leq j<k\leq p}Q(u_j, u_k)\Big).
\end{aligned}
\end{equation*}

 We will show that 
 for each $(z,t)\in Z\times T $
 there is $(u_1,\dots, u_p)\in Z^p$ such that 
\begin{equation*}
  f_\xi(u_1,\dots, u_p) =(z,t)
   \quad\text{and }\quad
  \abs{u_1}+\cdots+\abs{u_p} \le C(\abs{z}+\abs{t}^{1/2}).
\end{equation*}

Let  $(z,t)\in Z\times T$. Write the system
\[
  \Bigl(\sum_{j\le p}u_j\,,\,Q\Bigl(\sum_{i=1}^p(p-2i+1)u_i,\xi\Bigr)+
 \sum_{1\le j<k\le p}Q(u_j,u_k)
 \Bigr)=(z,t).
\]
 Since  we can not expect that the map is a submersion for all $\xi$, see Remark~\ref{presto},  we give a discussion independent of $\xi$ by restricting  ourselves to the vectors $u_j$ such that
$  \sum_{i=1}^p(p-2i+1)u_i=0$, or equivalently
\begin{equation}\label{uppi} 
 u_p=\sum_{j=1}^{p-1}\frac{p-2j+1}{p-1}u_j.                                           \end{equation} 
In such way, we are lead to work with a new  system on $Z^{p-1}$ 
of the form
\[\begin{aligned}
 &  \Bigl( \sum_{j\le p-1}u_j+\sum_{j\le p-1}  
 \frac{p-2j+1}{p-1}u_j \, ,\, 
 \sum_{j<k\le p-1}Q(u_j,u_k)+
 Q\Bigl(\sum_{k\le p-1}u_k,\sum_{j\le p-1}
 \frac{p-2j+1}{p-1}u_j
 \Bigr)\Bigr)
\\&=(z,t).
 \end{aligned}
\]
The $Z$-component is
$
 \sum_{j\le p-1}\frac{2(p-j)}{p-1}u_j =z
$
and, after some calculation we  see that the $T$-com\-po\-nent is 
\[
\begin{aligned}
&\sum_{j<k\le p-1}Q(u_j,u_k)+
 Q\Bigl(\sum_{k\le p-1}u_k,\sum_{j\le p-1}
 \frac{p-2j+1}{p-1}u_j
 \Bigr)
 \\&=\sum_{j<k}\Bigl\{
 1+\frac{p-2k+1}{p-1}-\frac{p-2j+1}{p-1}
 \Bigr\}Q(u_j,u_k)
 \\&
 =\frac{1}{p-1}\sum_{j<k\le p-1}(p+2j-2k-1)Q(u_j,u_k).
\end{aligned}
\]
Ultimately, we get the system
\[
 \left\{\begin{aligned}
         &\sum_{j\le p-1}(p-j)u_j=\frac{p-1}{2} z
         \\&\sum_{j<k\le p-1}(p+2j-2k-1)Q(u_j,u_k)=(p-1)t.
        \end{aligned}
\right.
\]
From the first equation we find 
\begin{equation}\label{upimenouno} 
u_{p-1}=\frac{(p-1)z}{2}- \sum_{j\le p-2}(p-j) u_j.
\end{equation} 
Therefore, the second equation becomes
\begin{equation*}
\begin{aligned}
& \sum_{j<k\leq p-2}(p-1-2(k-j))Q(u_j, u_k) 
+
 \sum_{j\leq p-2}(-p+1+2j) Q\Big( u_j, \frac{p-1}{2}z -\sum_{k\leq p-2} (p-k) u_k \Big)
\\&\qquad =(p-1)t.
\end{aligned}
\end{equation*}
After some simplifications we get
\begin{equation*}
\begin{aligned}
 Q\Big (
 \sum_{j\leq p-2} (-p+1+2j)u_j,\frac{p-1}{2}z
 \Big )
 + (p-1)\sum_{j<k\leq p-2}(1+k-j) Q(u_j, u_k) =(p-1)t.
\end{aligned}
\end{equation*}
Next we eliminate the first term by choosing
$\sum_{j\leq p-2}(-p+2j+1) u_j=0$
or, in other words,     
\begin{equation}\label{upimenodue} 
u_{p-2}=\sum_{j\le p-3}\frac{p-2j-1}{p-3}u_j .                                                         \end{equation} 
We get
\begin{equation*}
\begin{aligned}
 t&=
 \sum_{j<k\leq p-3}(1+k-j)Q(u_j, u_k) +\sum_{j\leq p-3}(p-1-j)Q\Big( 
 u_j,
 \sum_{k\leq p-3}\frac{p-2k-1}{p-3}u_k
 \Big)
\end{aligned}
\end{equation*}
which, after some manipulations  takes the simple form
\begin{equation}
 \label{mp10} 
\sum_{1\le j<k\le p-3}(p-3-2(k-j) )Q(u_j,u_k)= (p-3)t. 
\end{equation}

In order to solve the system \eqref{mp10}, we take $p-3=2\ell+1$ a sufficiently large odd number and we use Lemma~\ref{rick} below.  This ensures that for all $t\in T$ there are $u_1, \dots, u_{p-3}$ which solve  \eqref{mp10} and such that  $ |u_1|+\cdots +|u_{p-3}|\leq C|t|^{1/2}$.
Then, by \eqref{upimenodue} we find  $u_{p-2}$ with estimate $|u_{p-2}|\leq C|t|^{1/2}$. Furthermore, \eqref{upimenouno} gives the  value of $u_{p-1}$ and the estimate  
$
|u_{p-1}|\leq C(|z|+|t|^{1/2})$. 
 To conclude the proof, we get $u_p$ from \eqref{uppi} again evaluating with 
 $ |z|+|t|^{1/2}  $.
\end{proof}

In order to state and prove Lemma~\ref{rick}, define for $q\geq 3$
\begin{equation}
 \label{piqu}
 P_q(u_1,\dots, u_q):= \sum_{1\le j<k\le q}(q-2(k-j))  Q(u_j,u_k)
\end{equation} 
It is easy to see that  $P_2=0$ and  we agree to let  $P_1= 0$, if needed.

 \begin{lemma}[Properties of the function  {\eqref{piqu}}]
 \label{rick} 
Let $q=2\ell+1\geq 5$ be an odd number. Then, there is a linear change of variable $T_\ell: Z^{2\ell+1}\to Z^{2\ell +1}$, which we denote by $(u_1, \dots u_{2\ell+1})\mapsto (v_1, \dots, v_{2\ell +1})$
such that 
\begin{equation*}
 P_{2\ell+1}(u_1,\dots, u_{2\ell+1})=\sum_{k=1}^\ell
 \frac{2\ell+1}{2k+1}Q(v_{2k}, v_{2k+1}).
\end{equation*}
In particular, by Lemma~\ref{duetre},  if the H\"ormander condition holds, then the map $P_{2\ell+1}$ is open at the origin in the sense~\eqref{accaqu} for  sufficiently large $\ell$.
\end{lemma}
Note that the variable $v_1$ does not appear, after the change of variable.   An analogous version (not needed for our purposes) holds for $P_{2\ell}$ with $2\ell\geq 4$.

 \begin{proof}[Proof of Lemma \ref{rick}]
Let $q\geq 5  $ be an odd number.   
Let us 
make the linear change of variable  
  \begin{equation}\label{vacuum} 
\begin{aligned}
   v_{q-1}&=\sum_{j=1}^{q-1}(2j-q)u_j
\qquad \text{and}  \qquad  v_q&=u_q-\sum_{j=1}^{ q-2}\frac{-q+2+2j}{q-2}u_j
\end{aligned}  \end{equation} 
which leaves the variables $u_1,\dots, u_{q-2}$ unchanged. We have the following recursive relation
  \begin{equation}\label{eqq} 
P_q(u_1,\dots, u_q):= 
Q(v_{q-1},v_q)+\frac{q}{q-2}P_{q-2}(u_1,\dots, u_{q-2}).
  \end{equation}

 To prove \eqref{eqq}, write  
  \[
P_q(u_1,\dots,u_q)=   \sum_{1\le j<k\le q-1}(q-2k+2j)Q(u_j,u_k)+\sum_{j\le q-1}(2j-q)Q(u_j, u_q).
  \]
Letting $v_{q-1}=\sum_{j\le q-1}(2j-q)u_j$, or  equivalently 
$
 u_{q-1}=\frac{1}{q-2}v_{q-1}-\sum_{j\le q-2}\frac{2j-q}{q-2}u_j,
$
we can  eliminate  $u_{q-1}$, getting  
\[\begin{aligned}
&P_q (u_1,u_2, \dots, u_{q-2}, u_{q-1}, u_q)
\\&= \sum_{j<k\le q-2}(q-2k+2j)Q(u_j,u_k)
\\&\quad +Q\Bigl(
\sum_{j\le q-2}(-q+2+2j)u_j\,,\,
\frac{1}{q-2}v_{q-1}-\sum_{k\le q-2}\frac{2k-q}{q-2}u_k
\Bigr)+ Q(v_{q-1}, u_q)=:(*).
\end{aligned}
\]
To conclude the computations, we  insert the change concerning $v_q$  in~\eqref{vacuum} an we get
\[\begin{aligned}
&(*)= Q\Bigl(v_{q-1} ,u_q-\sum_{j\le q-2}\frac{-q+2+2j}{q-2}u_j\Bigr)
\\& +\sum_{j<k\le q-2}
 \Bigl\{
 (q-2k+2j)+\frac{(-q+2+2j)(q-2k)}{q-2}
 +\frac{(2j-q)(-q+2+2k)}{q-2}\Bigr\}Q(u_j,u_k) 
  \\&  =
  Q(v_{q-1},v_q)+\frac{q}{q-2}\sum_{j<k\le q-2}(q-2-2(k-j))Q(u_j,u_k),
\end{aligned}\]
which is the desired identity \eqref{eqq}.

Next we iterate formula \eqref{eqq} starting from $q=2\ell+1$ and we get
\begin{equation*}
\begin{aligned}
P_{2\ell+1}(u_1, \dots, u_{2\ell+1})&= Q(v_{2\ell}, v_{2\ell+1}) +  \frac{2\ell+1}{2\ell-1} 
P_{2\ell-1}(u_1, \dots, u_{2\ell-1})
\\&=
 Q(v_{2\ell}, v_{2\ell+1}) +  \frac{2\ell+1}{2\ell-1}    Q(v_{2\ell-2}, v_{2\ell-1}) +
  \frac{2\ell+1}{2\ell-3}P_{2\ell-3}(u_1, \dots, u_{2\ell-3}) 
  \\& \vdots  
\end{aligned}
\end{equation*}
When we encounter $P_q$ with   $q=3$, the iteration stops, because
\[\begin{aligned}
P_3(u_1, u_2, u_3) :& =Q(u_1, u_2)+Q(u_2, u_3)-Q(u_1, u_3)=Q(u_1, u_2)+Q(u_2-u_1, u_3)
\\&
=Q(u_1, u_1+v_2)+Q(v_2, u_3)= Q(v_2, u_3-u_1)=Q(v_2, v_3),                                                          \end{aligned}
\]
where we applied  the change of variable $v_2=u_2-u_1$ and $v_3= u_3-u_1$, i.e.~\eqref{vacuum} with $q=3$.
 
 The proof of the lemma is concluded.
 \end{proof}

  \subsection{Construction of the inner cone}\label{homo} 
Using the higher-order envelopes of 
Theorem \ref{opp}, we can prove Theorem~\ref{conogelato}.

Let $C\subset \G$  be a convex set in a two-step Carnot group   $\G$ with group law~\eqref{machegruppo}. 
Let us consider a point  $(z,t) \in  C $. Let $p\in\N$ be such that Theorem~\ref{opp} holds true.  Assume that there is $\xi\in Z$ such that   $(z,t)\cdot(p\xi, 0)\in \Int(C)$.
 Define  the map
\begin{equation*}
 \Gamma_{(\z,\t)}(u_1, \dots, u_p):=e^{u_p \cdot X} \cdots e^{u_2\cdot X}\cdot e^{u_1\cdot X}(\z,\t),\quad \text{ for all $u_1, \dots, u_p,\z\in Z$ and $\t\in T$.}
\end{equation*}
Note that $\Gamma_{(\z,\t)}(u_1, \dots, u_p)=(\z,\t)\cdot \Gamma(u_1, \dots, u_p)$, where $\Gamma=\Gamma_{(0,0)}$ is the function introduced in  \eqref{eppi}.
In these notation, if the hypotheses of Theorem~\ref{conogelato} are satisfied,   we have $(\z,\t):=\Gamma_{(z,t)}(\xi,\dots, \xi)\in\Int C$. Let $\r>0$ so that $B ((\z,\t),\r)\subset \Int(C)$, where $B$ denotes the subRiemannian ball. By continuity, there is $\e>0$  and $\r>0$ so that,    if $ (z',t')\in
\G$ and 
$u_1,\dots, u_p\in Z$ satisfy
\begin{equation}
\label{unozero} 
 \max\{d((z,t), (z',t')), \abs{u_j-\xi}: j=1,\dots, p \}<\e,                                        
 \end{equation} 
then
\begin{equation}
 \label{unouno}\Gamma_{(z',t')}(u_1, \dots, u_p)\in B((\z,\t),\r)\subset \Int(C). 
\end{equation}

 We are now ready to give the proof of Theorem~\ref{conogelato}.  
 
 \begin{proof}[Proof of Theorem \ref{conogelato}]
\step{Step 1.} If $u_j $ and $(z',t')$ satisfy \eqref{unozero}, then for all nonnegative $\la_1,\dots, \la_p$ satisfying $\sum \la_j=p$, we have 
  \begin{equation*}
\Gamma_{(z',t')}(\la_1u_1,\dots, \la_p u_p)
\in B((\z,\t),\r)\subset \Int(C).
  \end{equation*}

To prove Step 1, 
denote by $d$ the subRiemannian distance and 
write first 
 \begin{equation*}
\begin{aligned}
& d\Big(  \Gamma_{(z',t')}(\la_1u_1,\dots, \la_p u_p),
\Gamma_{(z,t)}(\xi,\dots, \xi)\Big)
\\&\leq 
d\Big(  \Gamma_{(z',t')}(\la_1u_1,\dots, \la_p u_p),
\Gamma_{(z',t')}(\xi,\dots, \xi)\Big)
+
d\Big(  \Gamma_{(z',t')}(\xi,\dots, \xi),
\Gamma_{(z,t) }(\xi,\dots, \xi)\Big).
\end{aligned}
 \end{equation*}
The second term can be estimated with   $C (| z'-z|^{1/2}+|t'-t|^{1/2})$, 
where the constant $C$  depends on the vector $\xi$ appearing in the hypotheses.

By the esplicit form of $\Gamma$ described in~\eqref{eppi}, the first term can be evaluated by
\begin{equation*}
\begin{aligned}
&d\Big((z',t')\cdot \Gamma_{(0,0)}(\la_1 u_1\dots, \la_p u_p),   
(z',t')\cdot \Gamma_{(0,0)}(\la_1 \xi \dots, \la_p \xi)
\Big)
\\  &\leq C\Big( \Big| \sum_{j\leq p} \la_j u_j -p\xi  \Big|
 +
 \Big|  
 \sum_{j<k\leq p}Q(\la_j u_j, \la_k u_k) +Q\Big( \sum_{j\leq p} \la_j u_j, p\xi\Big)
 \Big|^{1/2}\Big).
\end{aligned}
\end{equation*}
The first term is easily evaluated:
$
 \Big|\sum_{j}\la_j u_j -p \xi \Big| =\Big|\sum_{j}\la_j (u_j -  \xi) \Big|\leq C\e
$. To estimate the term under square root, note that   for all $j,k$ we have by 
homogeneity and anti-symmetry, 
\begin{equation*}
\begin{aligned}
| Q(\la_ju_j, \la_k u_k)|=\la_j\la_k |Q(u_j-u_k, u_k)|\leq C|\xi|\cdot |u_j-u_k|\leq C \e.
\end{aligned}
\end{equation*}
An analogous argument can be used to see that  $\Big|Q\Big( \sum_{j\leq p} \la_j u_j, p\xi \Big)\Big|\leq C\e $, where again the constant $C$ depends on $|\xi|$, which is fixed.
This ends Step 1.

\step{Step 2.} We show that for all $\lambda\in\left]0,1\right]$ and for all  $(z',t') \in C$ with $d((z,t), (z',t')) < \e$ we have 
\[
\{\Gamma_{(z',t')}(\la u_1, \la u_2,\dots, \la u_p)  :\abs{u_j-\xi} <\e\quad\text{for all $j\in\{1,\dots,p\}$}\}\subset C. 
\]

To accomplish Step 2 observe first that  if $|u_1-\xi|<\e$, then 
$\Gamma_{(z',t')}(p 
u_1, 0, 0, \dots, 0)  \in C$ (applying  Step 1 with $\la_1=p$ and $\la_j=0$ for $j\geq 2$). Furthermore the points 
$(z',t') $ and $\Gamma_{(z',t')}(p 
u_1, 0, 0, \dots, 0)  $ are aligned. Therefore, we have $\Gamma_{(z',t')}(\la u_1,0,0,\dots, 0 )\in C$ for all $\la\in[0,p]$.

Next, look at the aligned points  \[\Gamma_{(z',t')}(\la u_1, 0, \dots, 0)\in C
                                   \text{ and }
                                   \Gamma_{(z',t')}\big(\la u_1, (p-\la) u_2, 0, \dots, 0\big).
                                  \]
   The second point belongs to $C$ by Step 1 and the points are aligned. Then     
   \[  \Gamma_{(z',t')}(\la u_1, su_2, 0, \dots, 0)\in C, \text{  for all $s\in[0, p-\la]$  .}
   \] 
   Taking $s=\la $ (which 
   is less than $  p- \la$), we discover that $\Gamma_{(z',t')}(\la u_1, \la u_2, 0, 
   \dots, 0)\in C$. 
   
Iterating the procedure, we  easily conclude that $\Gamma_{(z',t')}(\la u_1, \dots, \la u_p)\in C$ and  Step 2 is accomplished.

\step{Step 3.} We finalize the construction of the inner cone. Let $(z',t')\in C$ with $d((z,t),(z',t')
)<\e$. Then, by the previous steps,   denoting by $\ell_{(z',t')}$ the left translation,  
\begin{equation*}
\begin{aligned}
 C &\supseteq \big\{
 \Gamma_{(z',t')}(\la u_1, \la u_2, \dots, \la u_p): |u_j-\xi|<\e\;\forall j ,\;\la\in[0,1]
 \big\}
 \\&
 = \ell_{(z',t') } \Big\{ \Gamma (\la u_1, \dots, \la u_p): |u_j-\xi|<\e\;\forall j,\quad \la\in[0,1]\Big\}
 \\&
 = \ell_{(z',t') } \bigcup_{\la\in[0,1]} \delta_\la \Big\{\Gamma(u_1,\dots, u_p):|u_j-\xi|<\e \Big\}  
 \supseteq \quad\text{by Theorem \ref{opp}}
 \\&
 \supseteq
 \ell_{(z',t') } \bigcup_{\la\in[0,1]} \delta_\la B_{\Eucl}(\Gamma(\xi,\dots, \xi), c_0\e^2).
\end{aligned}
\end{equation*}
Let $\delta_0$ be such that
$B_{\Eucl}(\Gamma(\xi,\dots, \xi), c_0\e^2) \supseteq B(\Gamma(\xi,\dots, \xi), \delta_0)$. Therefore  
\begin{equation*}
 \begin{aligned}  C&   =\ell_{(z',t') } \bigcup_{\la\in[0,1]}   B(\Gamma(\la \xi,\dots, \la\xi  ), \delta_0 \la )
=\bigcup_{\la\in[0,1]}B\Big( (z',t')\cdot(\lambda p\xi , 0),  \delta_0  \la \Big),
\end{aligned}
\end{equation*}
which is the required open cone. We used the known properties of the natural dilations $\delta_\la(z,t):=(\la z,\la^2 t)$ of the group. Namely, $\delta_\la B((z,t), r)=B(\delta_\la(z,t), \lambda r)$, for all $\la, r>0$ and $(z,t)\in\G$.

\step{Step 4.} In Step 3, we worked with  
$(z',t')\in C$.
To conclude, it suffices to approximate any desired $(z',t')\in \p C$ with $d((z',t'), (z,t)) < \e$ with a sequence  
$(z'_n, t'_n)\in C$ and the construction is concluded.  \end{proof}

 \section{Monotone sets and intrinsic epigraphs} 
   \label{custom74}

We start this section giving the details of the proof of the classification of monotone sets in M\'etivier groups. The idea of the proof  has been suggested to the author by R.~Monti. 
 All we need to know about M\'etivier groups is that for all $ P,Q$ in such kind of group $ \G$ there is a Heisenberg subgroup   $ H_0 $ such that the left coset $P\cdot  H_0=:H $ contains~$P$ and~$Q$.   A Heisenberg subgroup of a two-step group~\eqref{machegruppo} is a three-dimensional subgroup of the form $\{\Span\{(z,0), (\z,0), (0, Q(z,\z))\}$ for suitable independent $z,\z\in Z$ with $Q(z,\z)\neq 0$. See \cite{ArenaCarusoMonti12,MontanariMorbidelli15} and the references therein.
\begin{proposition}\label{decimo} 
 If $E\subset\G$ is a precisely monotone set in a M\'etivier group, and if $\p E\varsubsetneq \G$, then 
   either   $  E\in\{\varnothing, \G\} $, or   there is an open  half-space $
 \Omega\subset\G$ such that $\Omega \subset E\subset \ol\Omega$.
\end{proposition}
Unfortunately we can not claim that    horizontally monotone sets in $\H$ are  Euclidean convex. Otherwise, the proof below would be much shorter. As an example, consider the set 
\begin{equation*}
 E=\{(x,y,t)\in\H: y\geq 0\}\setminus\{(x,0,0): x\in\R\}.
\end{equation*}
 \begin{proof} 
  Since $\p E\neq \G$, we have $\Int(E)\cup \Int(E^c)\neq \varnothing$. Assume that $\Int(E)\neq\varnothing$. Write then~$\G$ as   a disjoint union of $\Int E$ and $(\Int(E))^c=\ol{E^c}$. To prove the 
 proposition, 
 it suffices to show that   the nonempty set    $\Int ( E)$ is monotone in the Euclidean sense, i.e.~that $\Int(E)$ and $\ol{E^c}$ are convex in the usual sense. Then the statement follows from \cite[Lemma 4.2]{CheegerKleiner10}.
 
 Let $P$ and 
  $Q\in \Int E$ and let   $H :=P\cdot H_0 $  be the left coset of a Heisenberg subgroup    of~$\G$ containing~$P$ and~$Q$ (which always 
  exists, by the M\'etivier condition). Note that $  
  E\cap H $ is monotone in~$H$ and also $P,Q\in \Int_{H} (E\cap H)$,  which,  by 
  \cite{CheegerKleiner10},  is    either   an  open
  halfspace in~$H $,   or the whole~$H$    (we denoted by $\Int_{H}$ the interior in the induced topology on $H
  $). Thus, the segment $[P,Q]$ which connects~$P$ and~$Q$ is 
  contained in $\Int_{H} (E\cap H)$ which is a subset of  $ E\cap H$. Ultimately, $[P, Q]
  \subset E$. Applying the   same 
  argument to all pairs $P',Q'$ with~$P'$ close to~$P$ and~$Q'=P'+(Q-P)$ close to $Q$,  we conclude that   $
  [P, Q]
  \subset 
  \Int E$.
  
In order to show that the  set   $\ol{ E^c}$   
is Euclidean convex, assume that it is nonempty,  consider 
$P,Q\in  {E^c}$ and    let $H=P\cdot H_0$ be the left coset of a    Heisenberg subgroup $H_0$  such that $P,Q\in H$.   Since    $E^c\cap H  $ is monotone in $H$,   either $E^c\cap H=H$, and then $[P,Q]\subset E^c$, or    there is an open three-dimensional  half-space $\Sigma_0\subset H$ such that $ 
\Sigma_0\subset E^c\cap H \subset  \Sigma $, where $\Sigma \subset H$  is the  closure of~$\Sigma_0$.  Taking the closure in $\G$, we see that
$\ol{E^c}\cap H=\Sigma$ and thus $[P, Q]\subset   \Sigma =  
\ol{E^c}\cap H 
\subset \ol {E^c}$.

Finally, approximating points $P,Q\in \ol{E^c}$ with points in $E^c$, we get the Euclidean convexity of $\ol{E^c}$.
\end{proof}

Next we prove that if $\G=\H^n$ is the $n$-th Heisenberg group,  then the assumption $\p E\varsubsetneq \H^n$ can be removed. Recall that $\H^n=\C^n\times\R  \ni(z,t)= (x+iy,t)$ is equipped with the law
\[
 (z,t)\cdot(\z,\t):=(z+\z, t+\t+ 2\Im(z\cdot\ol \z)),\quad\text {with $z\cdot\ol\z=\sum_{j=1}^n z_j\ol \z_j$}.
\]

\begin{proposition}
 \label{vameglio} If $E\subset\H^n$ is a precisely monotone set  and $\varnothing\neq E\neq \H^n$, then there is an open  half-space $\Omega\subset\H^n$ such that $\Omega \subset E\subset \ol\Omega$.
\end{proposition}
 Unfortunately we have not been able to get a generalization of Proposition~\ref{vameglio} to the setting of M\'etivier groups.
\begin{proof}
 \step{Step 1.} We start by observing  that given two open half-spaces~$\Omega_1$
 and~$\Omega_2\subset  \H^n$ there are two aligned points $P_1\in \Omega_1$ and $P_2\in \Omega_2$. To check this remark,  
 if $\p \Omega_1$ is a characteristic plane,
 then, left-translating its characteristic point 
  to the origin, we must have 
  $\Omega_1= \{ t>0\}  $ or $\Omega_1=\{t<0\}$; if instead  
  $\Omega_1$ has noncharacteristic boundary,  up to left translations it must be  defined 
  by the    inequality $\sum_{j=1}^n (a_j x_j +b_j y_j)>0$ for some $a, b\in\R^n$.

 In the first case, starting from the fact that any halfspace $\Omega_2$ must contain points $(\z,\t)$ with 
 $\z\neq 0$, it is
 easy to check that any such point $(\z,\tau) $ is aligned with a point with $t>0$.  This can be seen by 
 looking at the parametrized line $\gamma(s):= (\z,\t)\cdot (-is  \z,0)$ with large $s$.
  The case $\Omega_1=\{t<0\}$ is analogous (look at $\gamma(s)$ with $s\to -\infty$). 
 
 In the second case, given $(\z,\t) \in \H$, all points  
 $
  \gamma(s)=(\xi+i\eta,\tau)\cdot(s(0+i  b),0) 
 $
belong to $\Omega_1$ for large $s$.

\step{Step 2.} 
We show the proposition by induction. Assume that the statement is true for $\H^n$. Let $E\subset \H^{n+1}$ be a monotone set with $\varnothing\neq E\neq \H^{n+1}$.
By  Proposition~\ref{vameglio}, it suffices to show that $\Int(E)\cup \Int(E^c)\neq\varnothing$.

For any  $z_{n+1}\in\C$  define   
$
H_{z_{n+1}}:=\{(z, z_{n+1}, t): (z,t)\in\H^n\}
$.
Assume without loss of generality that $H_0\cap E\neq \varnothing$. Then by inductive hypothesis
$H_0\cap E\supset \Omega_1$, where~$\Omega_1$ is an 
open $(2n+1)$-dimensional halfspace  contained in $H_0$.

There are two cases.
Either for all  $\z_{n+1}\in\C\setminus\{0\}$  we have  $  H_{\z_{n+1}}\subset E^c $, or  there is 
$\z_{n+1}\in\C\setminus\{0\}$ such that the set~$E\cap H_{\z_{n+1}}$ is nonempty and then it contains an open half space $\Omega_2\subset H_{\z_{n+1}}$.
The first case can not occour, because it would contradict the convexity of $E^c$. Then we are always in the second case.
By Step 1 we can find two points $P=(z, 0 , t)\in\Omega_1  $
and $Q=(z,0,t)\cdot (\z, \z_{n+1}, 0)\in\Omega_2$, where both $\Omega_1$ and $\Omega_2$ are contained in~$  E$.  

Let $\gamma(s)=(z,0,t)\cdot (s\z,s\z_{n+1},0)$ be a line connecting $\gamma(0)=P$ and $\gamma(1)=Q$. By convexity, $\gamma[0,1]\subset E$. Taking points $(z',t') $ in a small open neighborhood $G$ of $(z,t)\in\H_n$, 
we can consider the set of points
\[
\Sigma:= \{  (z',0,t')\cdot (s\z,s\z_{n+1} ,0) : s\in \left]0,1\right[  \quad 
(z',t')\in G\}.
\]
It is clear that  $\Sigma$ is  a $(2n+2)$-dimensional surface  contained in $E$. To conclude, take any point $P \in\H^{n+1}$ aligned with some point   $Q\in \Sigma $ and such that the segment $[P,Q]$ is transversal to $\Sigma$ at~$Q$. If $P\notin \p E$ we, immediately get $\p E\neq \H^{n+1}$ and Step 2 is concluded. If instead $P\in\p E$, then we can apply 
Proposition~\ref{trasversale} to  conclude that $\Int(E)\neq \varnothing$.  
\end{proof}

Next, we describe some properties  of monotone sets in a Carnot group of step two which follow from the results of the previous section.

First of all, as a corollary of Theorem~\ref{opp} we get  the generalization of \cite[Proposition~4.6]{CheegerKleiner10}   to general two-step Carnot groups.  
   
 \begin{proposition}\label{trasversale} 
Let $E\subset \G$ be a monotone set in a two step Carnot group~$G$ with law~\eqref{machegruppo}.  If a point $(z,t)$ belongs to a line $\ell$ and a surface 
   $\Sigma\subset E$ intersects transversally the line $\ell$ at a  second different point 
  $(\z,\t)\in \ell $,  then:
  \begin{enumerate}[nosep,label=(\roman*)]
   \item \label{unnis} If $(z,t)\in \ol  E$, then the open segment  $\left](z,t),(\z,\t)\right[\subset\ell$
   is  contained in $\Int (E)$.
   \item If $(z,t) \in E^c$, then the connected component of $\ell\setminus\{(\z,\t)\}$ 
   not containing $(z,t)$ is contained in $\Int (E)$. 
  \end{enumerate}
 \end{proposition}
 The first statement \ref{unnis} is a qualitative version of Theorem~\ref{conogelato} (it can also be proved by following the argument in \cite[Prop 4.6]{CheegerKleiner10} with the advice of changing   
 the map $Z\times Z\ni(u_1, u_2)\mapsto \Gamma_P(u_1, u_2)$ with the map   $Z^p\ni (u_1, \dots, u_p)\mapsto \Gamma_P(u_1, \dots, u_p)$  and using Theorem~\ref{opp}.  
As observed in \cite{CheegerKleiner10}, statement~\ref{unnis} holds for a merely  convex set $E$.
The second statement can be proved easily arguing as in \cite{CheegerKleiner10}.

Then we have the generalization of the following lemmas to   step 2 Carnot  groups $(\G, \cdot)$.
\begin{lemma}[compare \cite{CheegerKleiner10}, Lemma 4.8]\label{linea} Let 
$(\G, \cdot)$ be a two step Carnot group with law~\eqref{machegruppo}.
  If $(z,t)$ and $(\z,\t)=(z,t)\cdot (u,0)\in\G $ are   aligned points with $u\neq 0$
  and both belong to $\p E$, then the whole line  connecting them is contained in $\p E$.
 \end{lemma}
 \begin{proof}
  The same of \cite{CheegerKleiner10}.
 \end{proof}

 The following lemma should be compared with its Heisenberg version,  Lemma~4.9 in~\cite{CheegerKleiner10}.
  \begin{lemma}\label{unionline} 
If $E\subset\G$ is monotone, then for all $P\in \p E$  and for any 
two-dimensional subspace~$V\subset H_P$, there is a line $\ell$ satisfying $P\in\ell \subset V$ and completely contained in $\p E$.
 \end{lemma}

 \begin{proof}
  Assume  without loss of generality that $(0,0)\in\p E$. Let $u,v\in Z$ be a pair of orthonormal vectors. 
  Look at  the circle $S:=\{(u\cos \theta + v\sin\theta, 0):\theta\in\R\}$. We claim that there is a point $
  (z,0)\in S\cap \p E$. Indeed, if this would not be true, then either $S\subset \Int E$ or $S\subset \Int 
  (E^c)$. Assume the former and let $(z,0)$ and $(-z,0)$ be two points of $S$. Since they are aligned, their intermediate point $(0,0)$ would be in $\Int E$, giving a contradiction. 
 \end{proof}

 Next we recall the definition of \emph{intrinsic graph}.
 \begin{definition}[Intrinsic graph]
  Let $\G=Z\times T$ be the two step Carnot group  with law~\eqref{machegruppo}. Let $\xi\in Z$ be a unit 
  vector and let $\xi^\perp\subset Z$ be its orthogonal space. Let $W\subset \xi^\perp\times T$ be an open set in $
  \xi^\perp\times T$ containing the origin. Given a function $\psi:W\to \R$, the   $\xi$-graph  (and the corresponding epigraph and ipograph) associated    
  with $\psi$ are
\begin{equation*}
\begin{aligned}
 &\operatorname{graph}(\psi):= \big\{ (\eta, \t)\cdot (\psi(\eta,\tau)\xi, 0): (\eta,\t)\in W
 \big\}.
\\
& \operatorname{epi}(\psi):= \big\{ (\eta, \t)\cdot (s \xi, 0): (\eta,\t)\in W
\quad\text{and $ \psi (\eta, \t)<s<\infty $} \big\}\text{, and }
 \\&
 \operatorname{ipo}(\psi):= \big\{ (\eta, \t)\cdot (s \xi, 0): (\eta,\t)\in W 
\quad\text{and $-\infty <s< \psi (\eta, \t) $
   } \big\}.
\end{aligned}
\end{equation*}
 \end{definition}
Intrinsic graphs appear in Geometric Measure Theory in Carnot groups. See \cite{FSSC01,AmbrosioSerraCassanoVittone06,FranchiSerapioni16} and the references therein.

Here we show the following theorem which relates monotonicity and intrinsic graphs.  
 \begin{theorem}\label{grafici} 
  Let $E\subset \G$ be a monotone set in a Carnot group of step two. Assume that $(0,0)\in\p E$. Then, either the whole horizontal plane 
  $ H_{(0,0)} $ is contained in $\p E$, or there is  a vector $\xi  \in Z$, an open neighborhood~$W$ of the origin in  $ \xi^{\perp}\times T$ and a function $
  \psi:W\to  \R$ such that 
   \begin{equation}
\operatorname{epi}(\psi)\subset \operatorname{int} (E)  ,
\qquad \operatorname{ipo}(\psi)\subset \operatorname{int} (E^c)\quad\text{ and }
\operatorname{graph}(\psi)\subset  \p E .
   \end{equation}
 Furthermore, the function $\psi$ is continuous. 
 \end{theorem}
 
 \begin{remark} Observe the following facts.
 
 \begin{enumerate}[nosep,label=(\alph*)]
  \item 
The inner cone  property is not global. For example let us consider 
the standard Heisenberg group   $\H=\{(x,y,t
 )\in\R^3\}$ with group law $
                             (x,y,t)\cdot(\xi, \eta,\t):=(x+\xi, y+\eta, t+\t+2(y\xi-x \eta))
                            $.
Look at the monotone set $E=\{x>t\}\subset\H$, where $\H$ denotes the standard Heisenberg group. There is no $\lambda>0$ such that the cone $\cup_{0<s<\infty }B((s,0,0), \lambda s)$ is contained in $E$. However, the truncated cone of the form 
 $\cup_{0<s\leq s_0 }B((s,0,0), \lambda s)$ is contained in $E$ for suitable $s_0$.
 
\item  The same example tells that the set $W$ appearing in the statement of Theorem \ref{grafici} can be a strict subset of the whole $\xi^\perp\times T$.  Indeed,  
 in the example above, $(0, 1/2, 0)\in\partial E$ becomes a characteristic point, but intrinsic graphs inherently  can not have characteristic points.

 \item The continuity of $\psi$ can be strengthened by saying that the function $\psi $ is \emph{intrinsic Lipschitz-continuous} (see~\cite{ArenaSerapioni09}).   
\item A version of Theorem~\ref{grafici} holds assuming that $E$ is  convex, but not necessarily monotone.

 \end{enumerate}
 
 \end{remark}

 \begin{proof}[Proof of Theorem \ref{grafici}]
Assume that   $(0,0)\in\p E$ and $(s_0e_1, 0)\in \Int(E)$ for some $s_0>0$.     
Split $(x,t)=(x_1,\wh x_1, t)$, where $\wh x_1= (x_2, x_3, \dots, x_m)\in\R^{m-1}$.
We   know 
by Theorem~\ref{conogelato}  that  there   is a positive number  $\lambda $  such that  
  \begin{equation*}
   \bigcup _{0<s\leq s_0}B((s, 0, 0), \lambda s) \subset \Int E, \quad \text{ and} \quad  
      \bigcup _{0<s\leq s_0}B((-s, 0, 0), \lambda s) \subset \Int (E^c).
  \end{equation*}

\step{Step 1.} To show the     existence of the function $\psi$,  
we prove that there is a neighborhood $W\subset \R^{m-1}\times \R^k$ such that for all $(\wh x_1, t)\in W$, we have   
  \begin{equation}\label{conetto} 
   (0, \wh x_1, t)\cdot (-s_0,0,0)\in \Int (E^c) \qquad\text{and}\qquad 
        (0, \wh x_1, t)\cdot (s_0,0,0)\in \Int( E).
    \end{equation}   To prove the second inclusion, since 
    $(0, \wh x_1, t)\cdot (s_0,0,0)= \Big(s_0, \wh x_1 ,t+Q\big((0, \wh x_1), (s_0, 0)\big) \Big)$
    it suffices to see that there is $W$ such that for all $
    (\wh x_1, t)\in W$, we have
\begin{equation*}
 \Big\|(-s_0, 0, 0)\cdot \Big(  s_0,  \wh x_1, t+ Q\big( (0, \wh x_1),   (s_0, 0)\big)
  \Big)  \Big\|< \lambda s_0  ,
\end{equation*}
which, using the group law and  the skew-symmetry of $Q$, becomes
\begin{equation*}
 \Big\| 
 \Big(0, \wh x_1 ,  t+ 2Q\Big((0, \wh x_1),  (s_0, 0) 
\Big)\Big\|< \lambda s_0  ,
\end{equation*}
which clearly holds true provided that $|\wh x_1|$ and $|t|$ are small.  We use the estimate $C|\wh x_1|
+C|t|^{1/2}+C|\wh x_1|^{1/2} s_0^{1/2} < \la s_0$ if $\wh x_1$ and $|t| $ are small enough. 

The proof of the first inclusion in \eqref{conetto} is analogous.

\step{Step 2.} By the properties of monotone sets, we can conclude that  for all $(  \wh x_1, t)\in W $, there is a number   $ \psi(\wh x_1, t)$ such that the point $(0, \wh x_1, t)\cdot (s,0,0)$ belongs to $\Int (E)$, if $s\in\left]\psi(\wh x_1, t), +\infty\right[$, to $\Int( E^c)$,  if $s\in\left]-\infty,\psi(\wh x_1, t)\right[$ and to $\p E$ if $s=\psi(\wh x_1, t)$. The function $\psi$ is  also 
bounded by $s_0$ on $W$.

\step{Step 3.} 
We show the continuity of $\psi$  at the origin  (at other points the same argument works). Namely we prove that
\begin{equation*}
\limsup_{(\wh x_1, t)\to (0,0)\in\R^{m-1}\times \R^k}\psi(\wh x_1, t)\leq 0.                                                                                                                                          \end{equation*}
The proof of $\liminf \geq 0$ is analogous.
Assume that  $\lim_{n\to \infty} \psi(\wh x_1^n, t^n)=L>0$, for some sequence $(\wh x_1^n, t^n)\to (0,0)$. Then  $(0,\wh x_1^n, t^n)\cdot (\psi (\wh x_1^n, t^n),0,0)\in \partial E$ for all $n\in\N$.  The sequence converges to $(0,0,L)$ and  this contradicts the fact that the half line $\{(s,0,0): 0<s<\infty \} $ is contained in $\Int (E) $.
\end{proof}

 \section{Monotone sets in \texorpdfstring{$\mathbb{H}\times \R$}{HxR}}
\label{4} 
 
 Consider  the direct product $ \H_{(x,y,t)} \times \R_u=\R^4$
with law
\begin{equation*}
 (x,y,u,t)\cdot (x' , y', u', t')  =   (x+x' ,  y+y',u+u', t+ t' +2(y x'-y'x )
\end{equation*}
and with the subRiemannian distance defined by the vector fields $X=\p_x+2y\p_t$, $Y=\p_y-2x\p_t$ and $U=\p_u$.

In this section we show the following result, which gives immediately the proof of Theorem~\ref{massimino}: 
\begin{theorem}\label{ungrafico} 
Let $E\subset \H\times\R$ be a monotone set, with $\varnothing\neq E\neq \R^4$. Then $\partial E$  is a hyperplane. 
\end{theorem}
  It is reasonable to think that a similar statement can be proved in the setting $\H\times\R^k$, with $k>1$. However, this would require some generalizations of the proof which are not completely trivial.

\subsection{Preparatory lemmas on monotone sets in \texorpdfstring{$\mathbb{H}\times \R$}{HxR}} 
\label{4.1} 
\begin{lemma}\label{frontier} 
 $\Int(E)\cup\Int (E^c) $ is a nonempty (open)  dense set, or equivalently,  $ \Int( \p E) = \varnothing $. 
\end{lemma}
\begin{proof}
Assume that $\p E\supset \Omega:=\{(x,y,u,t)\in\R^4: |x|,|y|, |u|, |t|<\e\}$, for some $\e>0$. We look at the set $\Sigma:= \{(x,y,0,t): (x,y,t)\in\R^3\simeq\H \}$. Since $E\cap\Sigma$ is monotone in $\H$, 
by~\cite{CheegerKleiner10},   there is $(z,0,t)\in \Omega$ and there is $\d>0$ such that $B((z,0,t),\d)\cap \Sigma   \subset   E\cap \Sigma$, or  $B((z,0,t),\d)\cap 
\Sigma \subset  E^c\cap \Sigma$. Assume the former and look at the point $(z,\frac{\e}{2}, t)\in \p E$ by assumption. Then,  by Proposition~\ref{trasversale} we see that $(z,u,t)\in \Int(E)$ for all $u\in\left]-\infty,\frac{\e}{2}\right]$. In particular,  $(z,0,t)\in \Int E$. We have a contradiction and the lemma is proved. 
\end{proof}

 \begin{lemma}\label{giroconto} 
  Let  $P\in \H\times \R$  be such that the horizontal plane $H_P$ is contained in $\p E$, then $\p E= H_P$.
 \end{lemma}
\begin{proof}
 Let $H_0=\{(\xi,\eta,\omega, 0):(\xi,\eta,\omega)\in\R^3\}\subset \p E$ and assume by contradiction that a point $P_0=(x_0,y_0,u_0,t_0)$ with $t\neq 0$ belongs to $ \p E$. Without loss of generality assume $P_0 =(x_0,y_0,0, 1)$. 
 
\step{Step 1. } We claim that we can assume $z_0=(x_0, y_0)\neq (0,0)$. Indeed, if $P_0= (0,0,0,1)\in\p E$, we look at the plane $S=\{ 
 (\z,0,1):\z\in\C\}$. By Lemma~\ref{unionline},  there is a line $\{(sz_0, 0,1):s\in\R\}\subset \p E$.
  Up to  a rotation we may assume $z_0=(x_0,0)$.  

 \step{Step 2.} 
Next we show that  $\p E=\G$, contradicting Lemma~\ref{frontier}.
  To see the claim,
  let $P_0=(x_0,  0,0,1) \in \p E$ be the point found in Step 1. 
  Look at the family of lines
\begin{equation*}
 \gamma_{\xi,\eta,\s}(\la)=(x_0,0,0,1)\cdot (\la\xi,\la\eta,\la\s, 0)=(x_0+\la\xi, \la\eta, \la\sigma,1-2x_0\la\eta),
\end{equation*}
where $(\xi,\eta,\s)\in\R^3$ is a nonzero vector.  It is easy to see that for any $\eta\neq 0$ the line $\gamma_{\xi,\eta,\s}$ touches the plane $t=0$  at a time $\la=\frac{1}{2x_0\eta}\neq 0$. This implies that the whole line is contained in $\p E$, by Lemma~\ref{linea}. Therefore, if we consider the horizontal plane at $P_0$
\begin{equation*}
 H_{P_0}=\Big\{ (x_0+\xi,\eta,\s, 1-2x_0\eta): \xi,\eta,\s\in \R\Big\}, 
\end{equation*}
all its points with $\eta\neq 0$ belong to $\p E$. Since $\p E$ is closed, we conclude that 
$H_{P_0}\subset\p E$. Note that $H_{P_0} $ is the plane of equation $t=1-2x_0 y$.
Now, consider the family of lines
\begin{equation*}
 \gamma_{x,y,u}(\la) = (x,y,u,0)\cdot(\la,0,0,0)= (x+\la, y, u, 2\la y).
\end{equation*}
where $(x,y,u)\in\R^3$. It is easy to see that   if $ \frac{1}{2x_0}\neq y \neq 0$, the line $ \gamma_{x,y,u}$
touches the plane $H_{P_0}$ at the nonzero time $\la=\frac{1-2x_0y}{2y}\neq 0$. By 
Lemma~\ref{linea}, this implies that for all  $(x,y,u)$ with $0\neq y\neq  \frac {1}{2x_0}$, the line $\gamma_{x,y,u}$ is fully contained into $\p E$. Thus
\begin{equation*}
\begin{aligned}
\p E &\supset  \Big\{ (x+\la, y, u, 2\la y): x,u,\la\in\R,\quad \text{and}\quad   0\neq y\neq  \frac{ 1}{2x_0}\Big\}.
\\&=\Big\{(\xi, \eta,\omega,\t): \xi,\omega,\t\in\R \quad\text{ and }  \frac {1}{2x_0} \neq \eta\neq 0
\Big\}.                                                                                        \end{aligned}
\end{equation*}
Taking the closure of the latter set, we conclude that~$\p E=\R^4$. 
This contradicts Lemma~\ref{frontier} and ends the proof.
 \end{proof}

 \begin{lemma}\label{giocattolino} 
Let $E  $ and $E^c$ be both nonempty. Let $0\in\p E$.  Then,  one,  and only one, of the following three  items holds.
\begin{enumerate}[nosep,label=(\roman*)]
 \item \label{giocherello}  $\p E=\{(x,y,u,0): (x,y,u)\in \R^3\} $.    
 
 \item \label{capperi}  The point $(0,0, 1,0)$ belongs to $\Int (E)\cup \Int (E^c)$.
 \item \label{cavoli} There is $(a,b)\neq (0,0) $ such that $(a,b,0,0)\in \Int (E)\cup \Int (E^c)$.

\end{enumerate}
 \end{lemma}

 \begin{proof}
  Assume that $\{ t=0\}\subset \p E$. Then  by Lemma \ref{giroconto} the boundary $\p E $ is the hyperplane of equation  $t=0$.

  If \ref{giocherello} does not hold, in order to show that at least one among \ref{capperi} and \ref{cavoli} must hold, assume that both such circumstances fail. This means that all points $(x,y,0,0)$
and $(0,0,u,0)$ belong to $\p E$ for all $(x,y,u)\in\R^3$. Therefore 
$
 \gamma(\la)= (\la x,\la y, (1-\la )u,0)\in \p E$ for all $x,y,u,\la \in\R$.
Taking $\la=\frac 12$ and since $x,y$ and $u$ are arbitrary, we discover that \ref{giocherello} 
holds,  getting  a contradiction. 
\end{proof}

\subsection{Classification of monotone \texorpdfstring{$(pX+qY)$-epighraph}{HxR}}
\label{4.2}  
In this section we analyze case \ref{cavoli} of Lemma~\ref{giocattolino}. 
 \begin{proposition}[Monotone $(pX+qY)$-epigraphs (local statement)]\label{parte_facile}
 Assume that  the origin $ (0,0,0,0)  \in \p E$ and $(p,q,  0,0)\in \Int (E)$ for some $(p,q)\neq (0,0)$.   Then there is a neighborhood $\Omega$ of $0$ in $\R^4$ and a plane $\Sigma$ of equation $rx+sy+cu+bt=0$, not 
containing $(p,q,0,0)$ such that $\p E\cap\Omega=\Sigma\cap \Omega$.  
 \end{proposition}
If we assume without loss of generality that $(p,q)=(1,0)$, then the plane has equation $x=ay+cu+bt$.  

 \begin{proof}[Proof of Proposition~\ref{parte_facile}]

 Whithout loss of generality, we prove the statement in the case $(p,q)=(1,0)$.  By Theorem \ref{grafici}, 
 there is a neighborhood~$U$
 of $(0,0,0)\in\R^3$ and 
 a continuous function $\psi(y,u,t)$ on~$U$ such that 
 \begin{equation}\label{crullo} 
\begin{aligned}
&  \operatorname{epi}(\psi):= \Big\{ (s, y, u, t+2y s) : (y,u,t)\in U \quad s>\psi(y,u,t)  \Big\}
\subset \Int (E),
 \\ &
 \operatorname{graph}(\psi):=\Big\{ (\psi(y,u,t), y, u, t+2y \psi(y,u,t)) : (y,u,t)\in U \Big\}\subset\p E,
\end{aligned}
 \end{equation}
while $\operatorname{ipo} (\psi)\subset \Int (E^c)$.

\step{Step 1.} We show that for all $\ol u$ close to $0$, there are  $a(\ol u),$ $  b(\ol u)$, and $ c(\ol u)\in\R$ so that 
\begin{equation}\label{gollone}
\begin{aligned}
 \{(x,y,\ol u, t) & \in\R^4:   x >  a(\ol u) y   + b(\ol u)t + c(\ol u)\}
  \subset \Big(  E \cap \{u=\ol u\}\Big)
 \\ 
\subset & \{(x,y,\ol u, t)\in\R^4: x\geq   a(\ol u) y   + b(\ol u)t + c(\ol u)
 \}. 
\end{aligned}
\end{equation}

To accomplish the step, note that for all $\ol u\in \R$, the intersection $  E\cap \{ u=\ol u\} $ is a monotone    set in $\G_{\ol u}:=\{(x,y,\ol u, t):(x,y , t)\in\C\times\R\}$, which is isomorphic to the Heisenberg group.    Thus by \cite{CheegerKleiner10} there are two open    half-spaces  $\Sigma^+ ,\Sigma^-\subset \G_{\ol u}$ with a Euclidean plane as a  common boundary and such that 
$
 E\cap \G_{\ol u} \supset  \Sigma^+$  and $E^c\cap \G_{\ol u} \supset \Sigma^-$. It suffices to show that such half-spaces have the form~\eqref{gollone}.
Let   $m(\ol u)x 
 +a(\ol u) y   + b(\ol u)t + c(\ol u)< 0
$ be the inequality defining 
 $\Sigma^+$, where 
  $m(\ol u),$ $  a(\ol u),$ $ b(\ol u)$ and $c(\ol u)$ are real numbers.
By~\eqref{crullo}, for small $\ol u$, we know that $(\psi(0, \ol u, 0)+1, 0, \ol u, 0)\in\Int(E)$ and 
$(\psi(0, \ol u, 0)-1, 0, \ol u, 0)\in\Int(E^c)$. Therefore we may assume $m(\ol u)=-1$ and  Step~1 follows.

\step{Step 2.}  After some easy computations (omitted) involving comparison between~\eqref{gollone} 
and~\eqref{crullo},  we get
\begin{equation}\label{golletto} 
\begin{aligned}
 \psi(y,u,t)= \frac{a(u)y + b(u) t + c(u)}{1-2 b(u) y}\qquad \text{for all $y,u,t$ close to $0$.}
\end{aligned}
\end{equation}

\step{Step 3.} The function $\psi$ is affine in $u$, i.e. 
  has the form
\begin{equation}\label{galeotto} 
 \psi(y,u,t)= A(y,t)+B(y,t)u,\qquad \text{for all $u,y,t$ close to the origin,}
\end{equation} 
where $A(y,t)$ and $ B(y,t)$ are suitable functions. 

To show this claim, 
let us look at  a pair of aligned points
\begin{equation*}
 P= (\psi(y,u,t), y, u, t+2y\psi(y,u,t))\quad \text{and } Q=
 (\psi(y,v,t), y, v, t+2y\psi(y,v,t))\in\p E
\end{equation*}
and at the  following line  containing $P$ and $Q$:
\begin{equation*}
 \begin{aligned}
\gamma (\s) & =(\psi(y,u,t), y, u, t+2y\psi(y,u,t))\cdot \big(\sigma[\psi(y,v,t)-\psi(y,u,t)], 0, \s
(v-u),0\big)
\\&=\Big( \psi(y,u,t)+\s[\psi(y,v,t)-\psi(y,u,t)], y, u+\s(v-u),
\\&\qquad \qquad t+2y\psi (y,u,t)  +2\s y[\psi(y,v,t)-\psi(y,u,t)]\Big).
\end{aligned}
\end{equation*}
Since $\gamma(0)$ and $\gamma(1)$ belong to $\p E$,  $\gamma(\sigma)\in \p E$  for all $\s\in \R$, by Lemma~\ref{linea}. If $\sigma\in[0,1]$ we are close to the origin, and the point $\gamma(\s)$ should belong to the graph of $\psi$. Then it must have the form 
$(\psi(y',u', t'), y',u', t'+2y'\psi(y',u', t'))$. Comparing the four coordinates, we discover that $y'=y$, $u' = u+\s(v-u)$, $t'=t$ and ultimately
\begin{equation*}
 \psi(y, u+\s (v-u),t)=\psi(y,u,t)+\s(\psi(y,v,t)-\psi(y,u,t)),
\end{equation*}
for all $y,u,v,t$ close to the origin and $\s\in[0,1]$. 
Thus~$\psi$ is affine in $u$ for all fixed~$y,t$.

\step{Step 4.} We show that for suitable $a_0,a_1,b,c\in\R$, we have
\begin{equation}\label{coccobello} 
 \psi(y,u,t)= \frac{(a_0+a_1 u)y+ b  t + cu}{1-2by}.
\end{equation}
To show such statement, start from identity
\begin{equation}\label{idess} 
\psi(y,u,t)= \frac{a(u)y + b(u) t + c(u)}{1-2 b(u) y}
=A(y,t)+ B(y,t)u
\end{equation} 
for all $y,t,u$ close to the origin. By linearity in $t$, we can write the right-hand side as 
$f_0(y)+f_1(y)t +[g_0(y)+g_1(y)t]u$, where the form of $f_0, f_1$ can be obtained letting $u=0$ in 
\eqref{idess}. Ultimately, \eqref{idess} can be written in the form
\begin{equation}\label{sss} 
\psi(y,u,t)= \frac{a(u)y + b(u) t + c(u)}{1-2 b(u) y}
=\frac{a_0 y+b_0t}{1-2b_0y}+ [g_0(y)+g_1(y)t]u,
\end{equation} 
where $a_0:=a(0)$, $b_0:=b(0)$ and $c_0:=c(0)=\psi(0,0,0)=0$. 
Letting $y=t=0$ we discover 
that $c(u)= g_0(0)u=: c_1 u$. Evaluating also at $y=0$, we find 
\begin{equation*}
 b(u)t+   c_1u   = b_0 t+[c_1+g_1(0)t]u=:b_0 t+ c_1 u+ b_1 tu.
\end{equation*}
Therefore,
\begin{equation}\label{dentino} 
 \psi(y,u,t)=\frac{a(u)y+(b_0+b_1 u)t + c_1 u}{1-2(b_0+b_1 u)y}
 =\frac{a_0 y+b_0t}{1-2b_0y}+ [c_1+\wh g_0(y)+(b_1+\wh g_1(y))t]u,
\end{equation} 
where $\wh g_0$ and $\wh g_1$ are suitable functions vanishing
at $0$.
Evaluating the identity~\eqref{dentino} at $t=0$, comparing left-hand side and right-hand side, we see that $\wh g_0$ is rational and smooth at $0$. Expanding at the first order in $y$ both sides with $t=0$, we see that
\begin{equation*}
 (c_1u + a(u)y) (1+2(b_0+b_1 u)y + O (y^2))=
 a_0 y+O(y^2) + c_1 u + \wh g_0'(0)uy,
\end{equation*}
which gives the identity
$
 a(u)+ 2 c_1(b_0 + b_1 u)u= a_0 +\wh g_0'(0)u
$.
Therefore we conclude  that $a(u)=: a_0+a_1u-2b_1c_1 u^2$. Thus the left-hand side of \eqref{dentino} becomes
\begin{equation*}
  \psi(y,u,t)=\frac{(a_0+a_1 u-2b_1 c_1 u^2)y
  +(b_0+b_1 u)t + c_1 u}{1-2(b_0+b_1 u)y}.
\end{equation*}
But by Step 3  (see the right-hand side  of \eqref{dentino}), $\psi $ should be affine in $u$ 
for all fixed $y,t$. This forces $b_1=0$, completing the proof of  Step 4.

\step{Step 5.} \label{step5} 
We show that the coefficient $a_1$ in \eqref{coccobello} must vanish. To see that, we 
 first observe (omitted computations) that     locally  the $X$-graph of $\psi$ agrees with the Euclidean graph
\begin{equation*}
\Sigma:=\Big\{ \big((a_0+a_1u)y +c u+bt, y,u,t \big) : \quad \text{$y,u,t$ close to the origin }\Big\}.
\end{equation*}

Let $y\neq 0$ be a  number close to $0$ and consider a pair of points
\begin{equation*}
 P=((a_0+c)y+a_1 y^2, y,y, 0)\quad \text{ and }\quad 
 Q=\big( -(a_0+c)y+a_1 y^2+bt, -y -y ,t\big).
\end{equation*}
Both $P$ and $Q$ belong to $\Sigma$. A computation shows that the choice $t=\frac{4a_1y^3}{1-2by}$ ensures that $P$ and  
\[
Q=\Big(  -(a_0+c)y+a_1y^2 \Big(1+\frac{4by}{1-2by}  \Big), -y,-y,\frac{4a_1 y^3}{1-2by} \Big)
\]  are aligned and both in~$\Sigma$.  Then the whole segment connecting $P$ and $Q$ belongs to $\Sigma$. In particular   
\begin{equation*}
 \frac{P+Q}{2}=\Big( \frac{a_1y^2}{1-2by},0,0, \frac{2 a_1 y^3}{1-2by} \Big)\in\Sigma.
\end{equation*}
Then $\frac{a_1 y^2}{1-2by} =b\cdot \frac{2 a_1 y^3}{1-2by} $, which forces $a_1y^2=0$ for all $y$ and ultimately  $a_1=0$.
\end{proof}

   \begin{proposition}[Monotone $(pX+qY)$-epigraphs (global statement)]\label{pappardella} 
   Let  the  hypotheses of Proposition \ref{parte_facile}  with $(p,q)=(1,0)$ be satisfied and let $x=ay+bt+cu $ be the equation 
   of the plane $\Sigma $ appearing in   Proposition~\ref{parte_facile}. Then we have $\Sigma=\p E$. 
   \end{proposition}

Before the proof, recall  the following geometric property 
concerning a given pair of lines~$\ell_1$ and~$\ell_2$  in the Heisenberg group parametrized by $ \Gamma_1(s)=(z_1, t_1)\cdot (s\z _1, 0)$ and $\Gamma_2(s)=(z_2, t_2)\cdot (s\z _2, 0)$.
See  the discussion in \cite[p.~343]{CheegerKleiner10}. If 
$\ell_1$ and $\ell_2$ are   \emph{skew} or \emph{parallel with distinct projection}, then for all $\sigma\in\R$,  except at least a singular value $\s_0$, there is $s_\s\in\R$ such that $\Gamma_1(\s)$ and $\Gamma_2(s_\s)$ are aligned.  

For the sake of completeness, recall  that two lines~$\ell_1$ and~$\ell_2$  parametrized by $ \Gamma_1(s)=(z_1, t_1)\cdot (s\z _1, 0)$ and $\Gamma_2(s)=(z_2, t_2)\cdot (s\z _2, 0)$ are \emph{parallel} if $\z_1 $ and $\z_2$ are linearly dependent and nonzero. They are \emph{parallel with distinct projection} if they are parallel and their  projections on the plane $x,y$ are (parallel and)   different lines in the plane. They are \emph{skew} if they are not parallel and have empty intersection.

   \begin{proof}[Proof of Proposition \ref{pappardella}]
    Assume without loss of generality that the equation of $\Sigma$ is $x=cu+bt$. This can be obtained after 
    a rotation in the variables $(x,y)$. We know by Proposition~\ref{parte_facile} that there is ~$\e>0$ such that
\begin{equation}\label{iona} 
   \begin{aligned}
A:=&  \{( cu+bt,y, u, t ) : |y|, |u|, |t|<\e  \}\subset \p E \quad\text{ and}                                                                                         
\\&     
 \{( cu+bt,y, u, t )\cdot(s,0,0,0): |y|, |u|, |t|<\e\quad 0<s<+\infty\}\subset \Int(E),                                                                                                                                                                                                                                                                                 
     \end{aligned}
\end{equation} 
while the same set with $s<0$   is  contained in $\Int (E^c)$. 

\step{Step 1.} Consider points of the form  $(bt, 0, 0, t)\in \Sigma$. These points are in $\p E$ as soon as $|t|<\e $. Let $\eta,\omega\in \R$ with $\abs{\eta}<\frac{1}{2b}$ and look at the line
\begin{equation}\label{lalli} 
\begin{aligned}
 \gamma_1(\s)&= (bt, 0, 0, t)\cdot ([c\omega-2b^2 t\eta]\s, \sigma\eta,\sigma\omega, 0)
 \\&= (bt(1-2b \eta\s)+c\s\omega, \s\eta,\s\omega, t(1-2b\eta\s) ),\quad\text{with }\s\in\R.
\end{aligned}
\end{equation}
If $\eta$ and $\omega$ are  close to zero,   we have $\gamma_1(0)$ and $\gamma_1(1)\in \p E$. Therefore $\gamma_1(\R)\subset \p E$. Observe also that, by~\eqref{iona}, if $t,\eta,\omega,\s$ are sufficiently close to~$0$, the point~$ \gamma_1(\s)$ belongs to the small surface~$A$ appearing in \eqref{iona}.

\step{Step 2.} We claim that  $\p E$ is contained in $\Sigma$.

  Assume  by contradiction that a point 
$P= (\ol x, \ol y, \ol u, \ol t)$ belongs to $\p E\setminus \Sigma$. By Lemma~\ref{unionline}, there is~$(p,q )\neq (0,0 )$ such that   
\begin{equation}\label{radio} 
 \gamma_2(s):=(\ol x, \ol y, \ol u, \ol t)\cdot (ps, qs, 0,0)\in \p E\quad\text{for all $s\in\R$.}
\end{equation}
We discuss first the case $(b,c)\neq (0,0)$. In this case, it is easy to see that we can find $t\in \left]\e,\e\right[$, $\eta,\omega\in \R$ such that the vectors $(p,q)$ and 
  $(c\omega-2b^2 t\eta, \eta)$ appearing   
  in~\eqref{radio} and  in~\eqref{lalli} respectively,  are independent.  After this choice, if we 
  indicate by $\pi_\H(x,y,u,t)=(x,y,t)$, the lines  $\pi_\H\gamma_1$ and $\pi_\H \gamma_2$ are not parallel 
  in $\H$. Taking if needed  a small modification of $t$, we may assume that they are also skew. Therefore, by the property discussed before the proof, we can find $\s$ as close as we wish to the origin and a 
  corresponding $s_\s\in\R$ such that $\pi_\H \gamma_1(\s)$ and $\pi_\H\gamma_2(s_\s)$ are aligned. Denote 
  by $
  \gamma_3$ the line   connecting $\gamma_1(\s)=\gamma_3(0)$  and $\gamma_2(s_\s)=\gamma_3(1)$. Our line~$\gamma_3$ is contained in $\p E$,
  is transversal to $\Sigma$ and  touches $\Sigma$ at a point very close to the origin. This contradicts the second line of~\eqref{iona}  and  concludes Step~2, at least in the case $(b,c)\neq (0,0)$.  

In order to discuss Step 2 for $b=c=0$, note that in this case  $\Sigma $ is the plane $x=0$ and we can start from the inclusion
$
\{(0, y, u, t): |t|<\e,\; (y,u)\in\R^2\}\subset\p E
$. 
Let us consider the curve $\gamma_1(\s)=(0,\s,0,t)$, take $(\ol x,\ol y,\ol u,\ol t)\in\p E$ with $\ol x\neq 0$
and the corresponding line $\gamma_2 $
 of the form~\eqref{radio}. If $p=0$, then, since $\ol x\neq 0$, $\pi_\H\gamma_1$ and $\pi_\H\gamma_2$ are parallel with distinct projection and we can conclude as in the case $(b,c)\neq(0,0)$ above. If instead, $p\neq 0$, the lines $\pi_\H\gamma_1$ and $\pi_\H\gamma_2$ are not parallel. Then, changing if needed the  choice of $t$ we may assume that they are skew. In either case, we get the same contradiction of case $(b,c)\neq (0,0)$.

\step{Step 3.} We have $\Sigma=\p E$. To show this claim, note that we already know that $\p E\subset\Sigma $. If $\Sigma\neq \p E$, then $(\p E)^c $ is open and connected. Therefore, it is contained either  in $\Int (E)$ or in $\Int (E^c)$.  Assuming the former and looking at the line $\gamma(s)= (s,0,0,0)$, we see that  both $\gamma(1)$ and $ \gamma(-1)$ belong to $  \Int (E)$, getting a  a contradiction. 
    \end{proof}

  \subsection{Classification of monotone \texorpdfstring{$U$-epighraphs}{HxR}}
 \label{4.3} 
   In this section we discuss case \ref{capperi} of Lemma \ref{giocattolino}.
 \begin{proposition}[Monotone $U$-epigraphs (local statement)]\label{ugra} 
  Let $(0,0,0,0)\in\p E$ and let $(0,0,1,0)\in \Int E$. Then there are $\e>0$,
  a 
  neighborhood $W$ of the origin in $\R^3$,
  a linear function $\psi:W\to \R$ with $\psi(0,0,0)=0$
such that $  \operatorname{graph}(\psi) \subset\p E$, 
  while  $ \operatorname{epi}(\psi)\subset \Int (E) $
and   $ \operatorname{ipo}(\psi) \subset \Int(E^c) $.
 \end{proposition}
By the direct product structure, Euclidean and intrinsic $U$-graphs are the same.
As a by-product of the following proof, we get a classification of all ``horizontally affine'' functions $\psi:\H\to\R$, i.e. all functions    
such that $\psi((z,t)\cdot(\la\z,0))-\psi(z,t)=\lambda[\psi((z,t)\cdot( \z,0)-\psi(z,t)]$ for all $\la,z,t,\z$. This classification could have some independent interest.

 \begin{proof}By Theorem \ref{grafici} the boundary of $E$ is locally an  $U$-graph. 
Namely, there is a neighborhood $W$ of the origin in $\R^3$ such that
 \begin{equation}\label{garullo} 
\begin{aligned}
 & \{(x,y,\psi(x,y,t), t): (x,y,t)\in W\} \subset \p E,
 \\&  \{(x,y,u, t):  (x,y,t)\in W\quad  +\infty>u>\psi(x,y,t)\}\subset  \Int (E)
 \\&  \{(x,y,u, t):  (x,y,t)\in W\quad -\infty< u <\psi(x,y,t)\}\subset  \Int( E^c).
\end{aligned}
 \end{equation}
Let $\Omega=W\times \R$.   Clearly, any point $(x,y,u,t)\in\Omega$  
 belongs to one and only one of the three sets in the left-hand sides  above. 

  The proof is articulated into  three steps.

  \step{Step 1.}  We show that 
  \begin{equation*}
   \psi(z,t)= m(z)t+q(z),
  \end{equation*}
for suitable functions $m, q$ defined in a neighborhood of the origin in $\R^2$.

 Let $P=(x,y,\psi(x,y,t), t)=(z,\psi(z,t),t)\in\p E\cap\Omega$. Let us look at the horizontal plane
  \begin{equation*}
   H_P=\Big\{(x,y,\psi(x,y,t), t)\cdot (\xi,\eta, u,0)=P\cdot 
   (\xi,\eta, u,0): 
    (\xi,\eta, u)\in\R^3 \Big\}.
  \end{equation*}
  Inside this plane, given an angle $\theta$ consider the two-dimensional subspace
  \begin{equation*}
   A_\theta:= \Big\{ P \cdot (\la e^{i\theta}, u,0): 
    (\la,u)\in\R^2 \Big\}.
  \end{equation*}
  We claim that  $ A_\theta\cap \p E$ consists exactly of a line. More precisely, 
  there is $\alpha(x,y,t,\theta )\in\R$ such that
  \begin{equation*}
   A_\theta\cap \p E =\Big\{ (x,y,\psi(x,y,t),t)\cdot (s\cos\theta, s\sin\theta, s\alpha(x,y,t,\theta),0) 
   :s\in\R\Big\}. 
  \end{equation*}
To show the claim,  note that Lemma~\ref{unionline} ensures that   $A_\theta\cap \p E$  contains at least a line passing for~$P$.
We shall show that assuming that there are two different  lines both contained in~$A_\theta$ and both containing $P$, we get a contradiction.  Let $\ell_1$ and $\ell_2$ be such lines. None of them can  have the form $  \{P\cdot (0,0, s,0):s\in\R\}$, because we know that $P\cdot (0,0, s,0)\in \Int (E)$ for all  $s>0$. Then we can write  
\begin{equation*}
 \ell_j=\Big\{\gamma_j(s):=P\cdot (se^{i\theta}, \a_j s,0): s\in\R\Big\}\quad \text{for $j=1,2$,}
\end{equation*}
where $\a_1\neq \a_2$. Since $\a_1\neq \a_2$, we find for all $s\neq 0$ a couple of distinct aligned points $\gamma_1(s)$ and $\gamma_2(s)$. Then, the line   connecting them,
\[
 \ell_s =\{ (z,\psi(z,t),t)\cdot (se^{i\theta}, u,0):u\in\R\Big\}
\]
is entirely contained in $\p E$, by Lemma~\ref{linea}. Taking $u=1$ and $s\neq 0$ very close to $0$,  we find  points in $\p E$    as close as we wish to the point $(z,\psi(z,t)+1, t)\in\Int (E)$. This gives a contradiction and proves the claim.

To accomplish Step 1, fix $(x,y)=z$ close to the origin and consider $\theta\neq \phi$ ($\operatorname{mod}\pi$), and $\tau\neq t$. The  lines provided by the previous construction are
\begin{equation*}
\begin{aligned}
&\big\{  \gamma_3(s):=(z,\psi(z,t), t)\cdot (se^{i\theta}, s\alpha_{z}(t,\theta),0): s\in\R\big\}\subset\p E
\text{ and }
\\ &\big\{ \gamma_4(\s):= (z,\psi(z,\t), \t)\cdot (\s e^{i\phi}, \s\alpha_{z} (\t,\phi ) ,0): \s\in\R\big\}\subset\p E,
\end{aligned}
\end{equation*}
where we wrote $\a_z(t,\theta)=\alpha(x,y,t,\theta)$. Following the ideas of (the proof of) \cite[Theorem~1.2]{Rickly06} and the similar argument in  \cite[Lemma~4.10]{CheegerKleiner10},
we look for pairs of aligned points  $\gamma_3(s)$ and $\gamma_4(\s)$. A short calculation shows that  $\gamma_3(s)$ is aligned with $\gamma_4(-\frac{\eta}{s})$, for all $s\neq 0$, where $\eta:=\frac{\t-t}{2\sin(\phi-\theta)}$. Fixed any  $s>0$ (also $s=1$ does the job), the line $\gamma_5$ connecting $\gamma_3(s)$ and $\gamma_4(-\eta/s)$ is
\begin{equation}\label{cucchi} 
\begin{aligned}
 \gamma_5(\la) &= (z,0,0)  \cdot(0, \psi(z,t), t)\cdot (se^{i\theta}, s\a_z(t,\theta), 0) 
 \\&
 \qquad \cdot \Big(-\la\Big ( \frac{\eta}{s}e^{i\phi}+s e^{i\theta} \Big), 
 \la\Big( \psi(z,\t)-\psi(z,t)-\frac{\eta}{s}\a_z(\t,\phi)-s\a_z(t,\theta)\Big),0 \Big)
 \\&=:(z,0,0)\cdot \Big(\wh z_5(\la), \psi(z,t)+\la [\psi(z,\t) -\psi(z,t)] +\wh u_5(\la), t+\la(\t-t)\Big),
\end{aligned}
\end{equation}
where 
\begin{equation}\label{klei} 
\begin{aligned}
\wh z_5(\la)&= se^{i\theta}-\la\Big(\frac{\eta}{s} e^{i\phi} +se^{i\theta} \Big),  \quad \wh u_5(\la) = s\a_z(t,\theta) -\la\Big(\frac{\eta}{s}\a_z(\t,\phi)+  s\a_z(t,\theta) \Big) ,
\end{aligned}
\end{equation}
while the form of the last coordinate follows from the choice of $\eta$. Note that both $\wh z_5(\la)$ and $\wh u_5(\la)$ change sign, if we change $s=1$ with $ s=-1$. Therefore  if we denote by $\gamma_6 $ the curve in~\eqref{cucchi} and~\eqref{klei} obtained with $s=-1$, the points $\gamma_5(\la)$ and $\gamma_6(\la)\in\p E$ are aligned for all~$\la\in\R$. Then their intermediate point belongs to~$\p E$   for all~$\lambda$. Ultimately
\begin{equation*}
\begin{aligned}
 \Big( z, \psi(z,t)+\la [\psi(z,\t) -\psi(z,t)] ,  t+\la(\t-t)\Big)\in \p E\quad \text{for all $\la\in\R$   and $t,\t$ close to $0$.}
\end{aligned}
\end{equation*}
In particular, choosing $\la = -t/(\t - t)$ for any pair $t\neq \t$, we see that 
\begin{equation*}
 \Big(z, \psi(z,t)-t\frac{\psi(z,\t)-\psi(z,t)}{\t-t}, 0  \Big) \in \p E
 \quad \text{for all $t\neq \t $ with  $t,\t$ close to $0$.}
\end{equation*}
Since for all $z$ close to the origin, the line $\{ (z, u,0):u\in\R\}$ contains only one point  $u=\psi(z,0)
\in \p E$, this implies that  
$\psi(z,t) = m(z) t +q(z)$ for suitable functions $m,q$ and for all $(z,t)$ close to the origin.

\step{Step 2.} We show that  
\begin{enumerate}[nosep,label=(\roman*)]
 \item \label{jojo} $q(\la z)= \la q(z)$ for  all small $  |z|$ and $|\la|\leq 1$;
 \item \label{kaka} $m(\la z)- m(0) =\la[m(z)-m(0)]$, for all small $|z|$ and $|\la|\leq 1$.
\end{enumerate}
 If $q(0)$ would not vanish, the function $q$ would satisfy \ref{jojo}, as $m$.

 To check \ref{jojo}, look at the points  $(z, \psi(z,0), 0)= (z, q(z),0)\in \p E$ for all $z$ close to the origin. Each of these points is aligned with $(0, q(0),0)=(0,0,0)$.
 Then, by Lemma~\ref{linea}, we have 
$  (\la z, \la q(z), 0)\in\p E$
 for all $\la\in\R$. Thus, from Theorem \ref{grafici}, for small  $|\la|$ and $|z|$, 
 we get $q(\la z)=\la q(z)$.

 To show property \ref{kaka} for  $m$, it suffices to consider the pair of   points
$P=(z, m(z)t+q(z), t)$ and $Q= (0, m(0)t,t)$, both in the boundary of $E$.
 The points $P,Q$ are aligned and  can be connected by
 \begin{equation*}
  \gamma(\la)= (\la z, m(0) t + \la[m(z)t+q(z)- m(0) t], t)\quad\text{ with } \la\in\R.
 \end{equation*}
 Thus, we get the identity
 $
  m(\la z )t+q(\la z) = m(0) t + \la[m(z)t+q(z)- m(0) t]
 $,
whose linear part in $t$ gives \ref{kaka}.

Note that, in spite of the fact that Theorem~\ref{grafici} is local, the functions $m$ and $q$ can be  defined globally, by their homogeneity property.

\step{Step 3.}
We show that $q(z)=q(x,y)= ax+by$ is linear and $m(z)$ is constant. 

 To accomplish Step 3, observe that for all $(z,t)$ and $(\z,0)$ close to the origin, the line
\begin{equation*}
 \begin{aligned}
\gamma(\la)&=\Big(  z,\psi(z,t), t\Big) \cdot \Big(\la\z, \la \big\{\psi((z,t)\cdot ( \z,0))-\psi(z,t)\big\},0\Big)
 \\&=\Big(z+\la\z, m(z)t+  q(z)+\la \big\{m(z+\z)(t+2\Im (z\ol \z) ) +q(z+\z)-m(z)t-q(z) \big\} ,
 \\&\qquad \qquad ,
 t+2\la\Im(z\ol \z\Big)
 \end{aligned}
\end{equation*}
satisfies $\gamma(0)\in\p E$ and $\gamma(1)\in\p E$. Then $\gamma(\la)\in \p E$ for all $\la\in[0,1]$, and therefore
\begin{equation}
\label{duepallini} 
\begin{aligned}
 &m(z+\la \z)(t+2\la\Im (z\ol \z))+q(z+\la \z)
\\& =
 m(z)t+q(z)+\la\Big[m(z+\z) (t+2\Im (z\ol \z))+ q(z+\z) -m(z)t-q(z)\Big].
\end{aligned}
\end{equation} 
Equating linear terms in $t$ gives
\begin{equation}
 \label{trepallini}
 m(z+\la\z)=m(z)+\la[m(z+\z)-m(z)],
\end{equation} 
for all $z,\z$ close to the origin and $\la\in[0,1]$, which means that $m$ is affine, i.e.~$m(x,y)=c+kx+hy$ for suitable constants $c,k,h\in\R$.
 To see this fact, fix $\e_0$ small, take $|x|,|y|<\e_0$ and write
\begin{equation*}
\begin{aligned}
 m(x,y)& = m(0,0)+ m(x,y)-m(x,0)+ m(x,0)-m(0,0)
 \\& 
 =m(0,0)+y\frac{[m(x,\e_0)-m(x,0)]}{\e_0}+x\frac{ m(\e_0,0)-m(0,0)}{\e_0}
 =:c+yk(x)+h x\end{aligned}
\end{equation*}
Then, the homogeneity of $z\mapsto m(z)-m(0)$ forces  $k(x)=$ constant.

Next  we look at \eqref{duepallini} with $t=0$  using the form   \eqref{trepallini} of $m(z+\la \z)$.
This gives 
\begin{equation*}
\begin{aligned}
 2\la\Im(z\ol \z)&\big\{ m(z)+\la[m(z+\z)-m(z)]\big\}
 +q(z+\la\z)
 \\&=q(z)+2\la\Im(z\ol \z) m(z+\z) +\la[q(z+\z)-q(z)].
\end{aligned}
\end{equation*}
This implies that
\begin{equation}\label{joaoss} 
  q(z+\la \z)-q(z)=\la \big(q(z+\z)-q(z) \big)
  +2\la(1-\la )\Im(z\ol\z)[m(z+\z)-m(z)].
\end{equation}
Taking two derivatives in $\la$, we conclude
 that
$m(z+\z)-m(z)$ must vanish for all $z,\z$ with $\Im (z\ol \z)\neq 0$ and by continuity of $m$, for all $z,\z\in\R^2$. Then $m$ is constant and  ultimately \eqref{joaoss} tells that $q$ is affine. 

This completes the proof of Proposition  \ref{ugra}. 
\end{proof}

Next we show the following global version  of Proposition \ref{ugra}. 

\begin{proposition}
 \label{ugra_global}  Let the hypotheses of Proposition \ref{ugra} be satisfied and let $u= ax+by+ct$ be the equation of the three-dimensional plane $\Sigma$  coming from Proposition~\ref{ugra}.
 Then~$\Sigma=\p E$.
\end{proposition}
\begin{proof}
Let $W=\{(x,y,t)\in\R^3: |x|,|y|,|t|<\e\}$ be a neighborood of 
the origin such that the statement of Proposition~\ref{ugra} holds true. 
\step{Step 1. }  We first show that $\p E\supset \Sigma\cap\{(x,y,ax+by+ct,t):|t|<\e$ and $(x,y)\in\R^2\}$.

To show this claim, given $t\in\left]-\e,\e\right[$ it suffices to look at the points $(0,0, ct,t)$ and $(x,y,ax+by+ct,t)$, where $|x|,|y|<\e$. Both points belong to $\p E$ and they  are aligned. Then,  
$
 \gamma(s)=(sx,sy, s(ax+by)+ct,t)\in\p E$, for all $s\in \R$. 
 Step 1 follows immediately.

\step{Step 2. }  We show that $\Sigma\subset\p E$, or in other words, for all $(\xi,\eta,\t)=(\z,\t)\in \R^3$  we have
 \begin{equation}\label{passodos} 
(\xi,\eta,a\xi+b\eta+c\tau,\t)\in \p E .
\end{equation} 

It suffices to show this statement for $\z\neq 0$. Let $\t\in\R$ and  consider the curve 
\begin{equation*}
\begin{aligned}
\gamma(\mu): & = \Big(\xi-\frac{\t\eta}{2\abs{\z}^2},
\eta+\frac{\t\xi}{2\abs{\z}^2}, a\Big(  \xi-\frac{\t\eta}{2\abs{\z}^2}\Big)
+b \Big( \eta+\frac{\t\xi}{2\abs{\z}^2}\Big), 0
\Big)
\\&\qquad \qquad 
\cdot
\Big(\mu\frac{\tau\eta}{2|\z|^2} , -\mu \frac{\t\xi}{2\abs{\z}^2}, \mu \Big[
a\frac{\tau\eta}{2|\z|^2} -  b \frac{\t\xi}{2\abs{\z}^2}+ c\tau
\Big],0\Big)
\\&= 
\Big(\xi- (1-\mu) \frac{\t \eta}{2|\z|^2},
\eta+(1-\mu)\frac{\t\xi}{2|\z|^2}, 
\\&
\qquad \qquad 
,a\Big[ \xi- (1-\mu) \frac{\t \eta}{2|\z|^2}\Big]
+ b\Big[\eta+(1-\mu)\frac{\t\xi}{2|\z|^2}\Big] 
+c\mu\tau ,\mu\tau\Big).
\end{aligned}
\end{equation*}
We have $\gamma(0)\in \p E$. Furthermore, if $\mu>0$ is so small that $\mu\t<\e $, then by Step 1 we have $\gamma(\mu)\in\p E$. Therefore,   $\gamma(\mu)\in\p E$ for all $\mu\in\R$. Taking 
$\mu=1$, we see that~\eqref{passodos} holds.

\step{Step 3.} We show that $\Sigma: = \{(x,y, ax+by+ct, t): (x,y,t)\in \R^3\}=\p E$. 

By Step 2, we know that $\Sigma\subset \p E$. Assume by contradiction that there is  $(\xi, \eta,  u,\t)\in \R^4\setminus \Sigma$ belonging to $\p E$. Then the whole line $\{(\xi, \eta, u,   \tau   ):u\in\R\}$ is contained in $\p E$. Up to a translation, we may assume that 
\begin{equation*}
 \Sigma=\{(x,y,\a x+\b y+\gamma t , t): (x,y,t)\in\R^3\}\subset\p E\quad\text{and } \ell=\{(0,0,u,0):u\in\R\}\subset \p E.
\end{equation*}
In particular, all points of the form  $P =(x,y, \a x+\b y ,0)$ and $Q= (0,0,u,0)$ belong to $\p E$ for all $x,y,u$. Since $P$ and $Q $ are aligned, we conclude that the intermediate point $(\frac{x}{2}
, \frac{y}{2}, \frac{\a x+\b y  +u}{2}, 0)$ belongs to $\p E$ for arbitrary $x,y,u\in\R$.  Then the plane  $t=0$ is fully contained in $\p E$ and Lemma \ref{giroconto} gives a contradiction. 
\end{proof}

\footnotesize
\def\cprime{$'$} \def\cprime{$'$}
\providecommand{\bysame}{\leavevmode\hbox to3em{\hrulefill}\thinspace}
\providecommand{\MR}{\relax\ifhmode\unskip\space\fi MR }
\providecommand{\MRhref}[2]{%
  \href{http://www.ams.org/mathscinet-getitem?mr=#1}{#2}
}
\providecommand{\href}[2]{#2}

\normalsize
 
 \end{document}